\documentclass{article}
\usepackage[utf8]{inputenc}
\usepackage[hscale=.7, vscale=.8]{geometry}
\usepackage{bbm, amsmath, amssymb, amsthm}
\usepackage{natbib}
\usepackage[colorlinks=true, citecolor={blue}]{hyperref}
\usepackage{mathtools}

\usepackage{subcaption}
\usepackage{graphicx, algorithm, algpseudocode}
\graphicspath{ {./figures/} }

\def\sh{\mathop{\text{sh}}\nolimits}

\def\Pb{\mathop{\mathbb{P}_{}}}

\def\Treen{\mathop{\mathbf{T}_n}\nolimits}

\def\tree{\mathop{\mathbf{t}_{}}\nolimits}
\def\shape{\mathop{\mathbf{s}_{}}\nolimits}

\newtheorem{theorem}{Theorem}
\newtheorem{lemma}[theorem]{Lemma}
\newtheorem{proposition}[theorem]{Proposition}
\newtheorem{corollary}[theorem]{Corollary}
\newtheorem{definition}[theorem]{Definition}
\newtheorem{example}{Example}
\newtheorem{remark}{Remark}

\title{Inference on the History of a Randomly Growing Tree}
\author{Harry Crane, Min Xu \\
Department of Statistics \\
Rutgers University}
\date{May, 2020}

\begin{document}

\maketitle
\begin{abstract}
The spread of infectious disease in a human community or the proliferation of fake news on social media can be modeled as a randomly growing tree-shaped graph. The history of the random growth process is often unobserved but contains important information such as the source of the infection. We consider the problem of statistical inference on aspects of the latent history using only a single snapshot of the final tree. Our approach is to apply random labels to the observed unlabeled tree and analyze the resulting distribution of the growth process, conditional on the final outcome. We show that this conditional distribution is tractable under a  \emph{shape-exchangeability} condition, which we introduce here, and that this condition is satisfied for many popular models for randomly growing trees such as uniform attachment, linear preferential attachment and uniform attachment on a $D$-regular tree. For inference of the root under shape-exchangeability, we propose $O(n\log n)$ time algorithms for constructing confidence sets with valid frequentist coverage as well as bounds on the expected size of the confidence sets. We also provide efficient sampling algorithms which extend our methods to a wide class of inference problems.
\end{abstract}

\section{Introduction}

Many growth processes, such as the transmission of disease in a human population, the proliferation of fake news on social media, the spread of computer viruses on computer networks, and the development of social structures among individuals, can be modeled as a growing tree-shaped network.  We visualize the process as a growing tree, as in Figure \ref{fig:illustration}(a), with each individual corresponding to a node labeled by its arrival time and associations between individuals represented by an edge between the corresponding nodes.  In the examples mentioned above, the edges of the tree may correspond, respectively, to the transmission of a disease, spread of a rumor, passage of a computer virus, or establishment of a friendship between the connected nodes.  For concreteness, we frame the following discussion in the context of disease spread, so that the edges of the tree represent a person-to-person spread of an infection.

In this setting, we assume there is an initial infected individual, called the {\em root}, at time $t=1$.  At each discrete time step $t=2,3,\ldots$, a new individual becomes infected by one of the individuals previously infected at times $1,\ldots,t-1$ according to a probability distribution that depends on the past history of the process.  By tracking the infection through the growth of the corresponding infection tree (as in Figure \ref{fig:illustration}(a)), the process of infection thus produces a sequence of trees $\tree=(\tree_1,\tree_2,\ldots)$ which represents the complete  history of infection in the population.  In particular, $\tree_n$ is a tree with $n$ nodes and a directed edge $uv$ indicating that node $u$ passed the infection to node $v$.  Equivalently, labels $1,\ldots,n$ can be assigned to the nodes of $\tree_n$ according to the time at which each node arrives, so that any edge $(u,v) \in\tree_n$ is immediately interpreted as a transmission from the node with the lower label to the one with the higher label.

\begin{figure}
    \centering
    (a) \includegraphics[scale=.4]{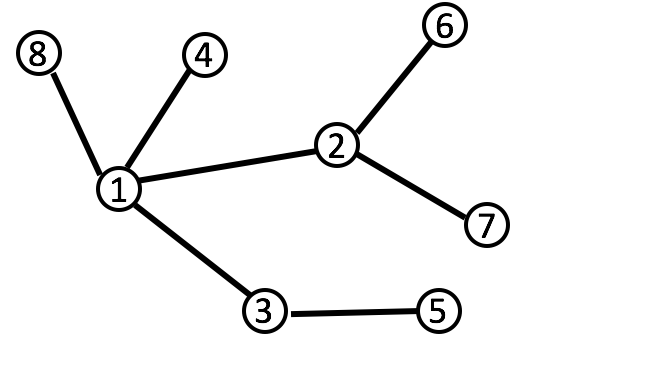}
    (b) \includegraphics[scale=.4]{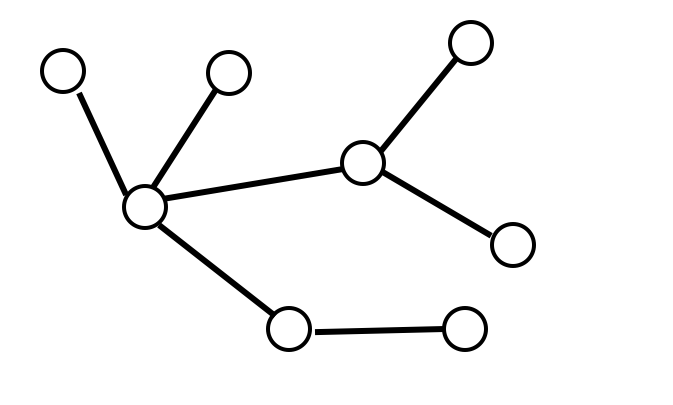}
    \caption{(a) A realization of a tree growth process with nodes labeled in order of arrival. (b) The shape of the process in (a) with node labels removed.  The main objective of the paper is to infer properties of the tree in (a) based only on the observed shape in (b).}
    \label{fig:illustration}
\end{figure}

In epidemiological applications, the infection history is important for enacting measures which prevent further spread, such as quarantining and testing.  But in many cases the complete infection history is not fully observable.  In the 2020 outbreak of Covid-19, for example, long incubation periods and asymptomatic spreading leads to incompleteness in the observed infection history (i.e., some or many edges in the infection tree are unobserved) as well as uncertainty about the direction of spread for those edges which are observed, as shown in Figure \ref{fig:illustration}(b).  Here we consider the problem of inferring properties of the disease history based on such partial observations of the infection spread.  Properties of special importance include inference of the root node (so-called `patient-zero') and inference of infection time.  We develop our methods in the specific context in which the contact pattern has been observed (i.e., the `shape' of the infection tree is known but its directions are not).  We discuss some relaxations to this assumption and limitations of our methods in Section \ref{section:discussion}.

Suppose that a disease has been transmitted among $n\geq1$ individuals (as illustrated in Figure \ref{fig:illustration}(a)) but only the contact pattern is observed (as in Figure \ref{fig:illustration}(b)).  We propose a methodological framework for answering general inference questions about the infection spread based on the observed contact pattern.  In the context of disease spread, important relevant questions include inference of so-called `patient-zero' (i.e., the initial source of infection) and also the time at which specific individuals became infected.  Our proposed method efficiently computes the conditional probability of the disease transmission history given the observed contact process.  These conditional probabilities enable us to construct valid confidence sets for the inference questions of interest.  These inference procedures may be applied in social media networks to help identify key spreaders and promoters of fake news or applied on the infection pattern of an epidemic (often obtained from contact tracing) 
to localize patient-zero or to reconstruct the chain of infection~\citep{hens2012robust, keeling2005networks}.





\subsection{Literature review}
Researchers in statistics~\citep{kolaczyk2009statistical}, computer science~\citep{bollobas2001degree}, engineering, and physics~\citep{callaway2000network} have studied the probabilistic properties of various random growth processes of networks, including popular models such as the preferential attachment model~\citep{barabasi1999emergence}. A line of research in the physics and engineering literature explores the problem of full or partial recovery of the network history based on a final snapshot \citep{young2019phase, cantwell2019recovering, timar2020choosing, sreedharan2019inferring, magner2018times}.  However, the problem of statistical inference on the history of a network growth process has been studied only recently.  In statistics, most existing work focuses on the problem of root estimation~\citep{shah2011rumors, shah2016finding, fioriti2014predicting, shelke2019source} and root inference~\citep{bubeck2017finding}. The latter work by \cite{bubeck2017finding} shows that one can construct a confidence set of the root whose size does not increase with the size of the network $n$. These directions are further developed by \cite{bubeck2017trees}, by \cite{lugosi2019finding} and \cite{devroye2018discovery}, who consider inference on a seed tree, by
\cite{shah2016finding}, who analyze situations where consistent root estimation is possible, and by~\cite{khim2017confidence}, who extend the results of \cite{bubeck2017finding} to the setting of uniform attachment on a $D$-regular tree. More recently, \cite{banerjee2020root} derives tight bounds on the size of confidence set of the root constructed from Jordan centrality. These results, although theoretically sound, do not give a practical method for inferring the root of a tree based on an observed contact pattern. For example, the confidence set algorithms in \cite{bubeck2017finding} and \cite{khim2017confidence} (described in detail in Section~\ref{sec:prev-work}) only give asymptotic coverage guarantees that hold under specific model assumptions and tend to be too conservative for practical purposes; see the numerical examples in Section~\ref{sec:simulation} for further discussion. 

\subsection{Summary of Approach}

In this paper, we address the above issues by proposing a new approach for inference on the history of a randomly growing tree. Our approach is to randomly relabel the nodes of the observed shape to obtain a random tree $\tilde{\mathbf{T}}_n$ whose labels are random but whose shape corresponds to that of the observed contact pattern. Random relabeling thus induces a latent sequence of subtrees $\tilde{\mathbf{T}}_1 \subset \tilde{\mathbf{T}}_2 \subset \ldots \tilde{\mathbf{T}}_{n-1} \subset \tilde{\mathbf{T}}_n$ which represents the history of $\tilde{\mathbf{T}}_n$. We study the conditional distribution of the history $\mathbb{P}(\tilde{\mathbf{T}}_1, \ldots, \tilde{\mathbf{T}}_{n-1} \,|\, \tilde{\mathbf{T}}_n)$ and show that the conditional distribution is tractable under a  \emph{shape-exchangeability} property, a distributional invariance property which is satisfied by the most common instances of the preferential attachment models including linear preferential attachment, uniform attachment and $D$-regular uniform attachment. 

For most of the paper, we focus on the problem of root inference, where our proposed method has a coverage guarantee for all $n$ and produces informative and computationally efficient confidence sets even for trees of millions of nodes. However, our approach is also applicable to a wide class of other inference problems such as arrival time inference and inference of the initial subtree given the final shape. 

\subsection{Summary of Methodology}

Although the problem of inferring the root node is conceptually simple, we need a number of technical definitions in the main paper to address subtle complications that arise from a formal analysis. To aid readers who are interested in a conceptual understanding, we first provide a short informal discussion of our methodology. 

For a given tree $\tilde{\mathbf{t}}_n$ and a node $u \in V(\tilde{\mathbf{t}}_n)$, we can think of a single realization of how the tree ``grew" from node $u$ as an ordered sequence of the nodes $\mathbf{u}_{1:n} = (u_1, \ldots, u_n)$ where $u_1 = u$. Not all ordered sequences are possible since the tree must be connected at all times and so we define $\text{hist}(u, \tilde{\mathbf{t}}_n)$ as the subset of allowable permutations of the nodes that starts at node $u$ and results in tree $\tilde{t}_n$ (see Figure~\ref{fig:random_label}). Interestingly, for a class of preferential attachment tree models including both uniform attachment and linear preferential attachment (Examples~\ref{example:uniform} and~\ref{example:linear-PA}), the conditional distribution of the growth realization $\mathbb{P}(\tilde{\mathbf{T}}_1, \ldots, \tilde{\mathbf{T}}_{n-1} \,|\, \tilde{\mathbf{T}}_n = \tilde{\mathbf{t}}_n)$ is uniform over the set of all allowable sequences $\cup_v \text{hist}(v, \tilde{\mathbf{t}}_n)$ (Proposition~\ref{prop:uniform-conditional} and Theorem~\ref{thm:pa-shape-exchangeable}). Moreover, the cardinality $\# \text{hist}(u, \tilde{\mathbf{t}}_n) $ can be computed in linear time (Proposition~\ref{prop:count-rooted-history}). Therefore, we can compute the conditional root probability $\mathbb{P}( \tilde{\mathbf{T}}_1 = \{u\} \,|\, \tilde{\mathbf{T}}_n = \tilde{\mathbf{t}}_n) = \frac{\# \text{hist}(u, \tilde{\mathbf{t}}_n)}{\sum_v \# \text{hist}(v, \tilde{\mathbf{t}}_n)}$, as the fraction of all allowable permutations that starts at node $u$, and construct a Bayesian credible set by taking the nodes with the highest conditional root probabilities. We prove that such a credible set also has valid Frequentist coverage at exactly the same level (Theorem~\ref{thm:frequentist-coverage}). \\

The paper is organized as follows.  In Section \ref{section:definition} we define the problem and review some earlier work.  In Section \ref{sec:approach} we describe our approach using random labeling and the shape exchangeability condition.  In Section \ref{section:general} we discuss general inference problems for observed contact patterns.  In Section \ref{sec:experiments} we show some simulation studies and illustrate our methods on flu data from a London school.



\subsection{Notation:} For $n \in \mathbb{N}$, we write $[n] := \{1, \ldots, n\}$ and $[n]_0:=\{0,1,\ldots,n\}$. We let $\mathbf{t} = (V,E)$ denote a labeled tree where $V$ is the finite set of nodes (generally taken to be $[n]$ where $n$ is the size of the tree) and $E \subset V \times V$ is the set of edges. 

For two finite sets $A, B$ of the same size, we write $\text{Bi}(A, B)$ as the set of all bijections between $A$ and $B$. For two labeled trees $\mathbf{t}, \mathbf{t}'$, we write $\pi \mathbf{t} = \mathbf{t}'$ for $\pi \in \text{Bi}(V(\mathbf{t}), V(\mathbf{t}'))$ if $(u, v)$ is an edge in $\mathbf{t}$ if and only if $(\pi(u), \pi(v))$ is an edge in $\mathbf{t}'$. In this case, we say that $\pi$ is an isomorphism between $\mathbf{t}$ and $\mathbf{t}'$. We note that for any labeled tree $\mathbf{t}$ and any bijection $\pi \in \text{Bi}(V(\mathbf{t}), V(\mathbf{t}))$, $\pi \mathbf{t}$ is always a valid tree that represents the result of relabeling the nodes of $\mathbf{t}$ by $\pi$. For a labeled tree $\mathbf{t}$ and a subset of nodes $B \subset V(\mathbf{t})$, we write $\mathbf{t} \cap B$ as the (possibly disconnected) subgraph of $\mathbf{t}$ restricted only to the nodes in $B$. 

Throughout the paper, we use bold upper-case letters such as $\mathbf{T}$ to denote a random tree and bold lower-case letters such as $\mathbf{t}$ to denote a fixed tree. For any two random objects $X, Y$, we write $X \stackrel{d}{=} Y$ if they are equal in distribution.

\section{Model and Problem Definition}\label{section:definition}

\subsection{Markovian tree growth processes}

Following the terminology used in discrete mathematics, we define a labeled tree $\mathbf{t}_n$ with $n$ nodes as a \emph{recursive tree} if $V(\mathbf{t}_n) = [n]$ and if the subtree $\mathbf{t}_n \cap [k]$, obtained by removing all nodes except those with labels $\{1, 2, \ldots, k\}$, is connected for every $k \in [n]$. In other words, any path from node $1$ to any other node must be increasing in the node labels.    Equivalently, we may view a recursive tree $\mathbf{t}_n$ as a map $[n]\to[n]_0$ with $\tree_n(j)=i$ indicating that the parent of $j$ is $i$, for $i<j$, and $\tree_n(j)=0$ indicating that node $j$ is the root of $\tree_n$. Figure \ref{fig:illustration}(a) illustrates a recursive tree with 8 nodes.  We write $\mathcal{T}^R_n$ to denote the set of all recursive trees with $n$ nodes. 

For recursive trees $\mathbf{t}_k,\mathbf{t}_{k'}$ with node labels $[k]$ and $[k']$, respectively, for $k<k'$, we write $\mathbf{t}_k\subset\mathbf{t}_{k'}$ if $\mathbf{t}_{k'}\cap[k]=\mathbf{t}_{k}$.  That is, $\mathbf{t}_{k}\subset\mathbf{t}_{k'}$ indicates that $\mathbf{t}_k$ is the subtree of $\mathbf{t}_{k'}$ obtained by removing all nodes and edges from $\mathbf{t}_{k'}$ except those labeled in $[k]$.  For $k<k'$, we call $\mathbf{t}_{k}$ and $\mathbf{t}_{k'}$ {\em compatible} if $\mathbf{t}_{k}\subset\mathbf{t}_{k'}$.  A family $(\mathbf{t}_n)_{n\geq1}$ such that $\mathbf{t}_k\subset\mathbf{t}_{k'}$ for all $1\leq k<k'$ is called {\em mutually compatible}.

A {\em Markovian tree growth process} is a mutually compatible family of random recursive trees $\mathbf{T}=(\mathbf{T}_n)_{n\geq1}$ such that $\mathbf{T}_n\in\mathcal{T}^R_n$ for each $n\geq1$ and
\begin{equation}
\label{eq:tps}
\Pb\bigl(\Treen=\mathbf{t}_n\mid(\mathbf{T}_{k})_{1\leq k\leq n-1}=(\tree_k)_{1\leq k\leq n-1}\bigr)=p_n(\tree_{n-1},\tree_n),\quad n\geq1,
\end{equation}
for some family of transition probabilities $(p_n)_{n\geq1}$. Any such process $\mathbf{T}=(\mathbf{T}_1,\mathbf{T}_2,\ldots)$ can therefore be constructed by sequentially adding nodes at discrete times $k=2,3,\ldots$ according to these transition probabilities and labeling the nodes of $\mathbf{T}_n$ by their arrival time.

\begin{example}\label{example:uniform}
For a straightforward example of a growth process, define
\[p_k(\mathbf{t}_{k-1},\mathbf{t}_{k})=\left\{\begin{array}{cc} 1/(k-1),& \mathbf{t}_{k-1}\subset\mathbf{t}_{k},\\
0,& \text{otherwise,}
\end{array}\right.\quad k\geq2,\]
so that new nodes attach to existing nodes uniformly at each time.  The resulting process is called the {\em uniform attachment (UA) model}. We note that in discrete probability, UA trees are called "random recursive trees" because it is uniform over the set of all distinct recursive trees (see e.g. \citet{drmota2009random}). We use the term "random recursive tree" in a more general sense in that we allow a random recursive tree to have any Markovian distribution. 
\end{example}

\begin{example}\label{example:linear-PA}
For another popular example, let $\text{deg}(w,\mathbf{t}_k)$ denote the {\em degree} of node $w$ in $\mathbf{t}_k$ and for $\mathbf{t}_{k-1}\subset\mathbf{t}_k$ let $w^*(\mathbf{t}_{k-1},\mathbf{t}_k)$ denote the unique node of $\mathbf{t}_{k-1}$ to which the node labeled $k$ connects to form $\mathbf{t}_k$.  We start with one node, attach another node to it, and then define transition probabilities by
\[
p_k(\mathbf{t}_{k-1},\mathbf{t}_k)=\left\{\begin{array}{cc}
=\frac{\text{deg}(w^*(\mathbf{t}_{k-1},\mathbf{t}_k), \mathbf{t}_{k-1})}{2(k-2)},&\mathbf{t}_{k-1}\subset\mathbf{t}_{k},\\
0,& \text{otherwise,}
\end{array}\right.\quad k \geq 3.\]
The resulting tree process is called the {\em linear preferential attachment model} (LPA), as nodes with high degree tend to accumulate more connections. In discrete probability, LPA trees are also referred to as random plane-oriented recursive trees (see e.g.~\citet{drmota2009random}). 
\end{example}

\begin{example}\label{example:PA}
To generalize the previous two examples, we define a general class of {\em preferential attachment (PA)} processes indexed by a fixed function $\phi \,:\, \mathbb{N} \rightarrow [0,\infty)$ as follows.  For short, we call these {\em $\text{PA}_\phi$ processes}. Given such a function $\phi$, we generate $\mathbf{T}_1$ as a singleton node with label $1$ and, for each $n = 2, 3, \ldots $, generate $\mathbf{T}_n$ as a random tree where we
\begin{enumerate}
\item choose an existing node $w \in [n-1]$ of $\mathbf{T}_{n-1}$ with probability proportional to $\phi(\text{deg}(w, \mathbf{T}_{n-1}))$, where $\text{deg}(w,\mathbf{T}_{n-1})$ denotes the degree of node $w$ in $\mathbf{T}_{n-1}$, 
\item and add edge $(n, w)$ to tree $\mathbf{T}_{n-1}$ to form $\mathbf{T}_n$.
\end{enumerate}
The resulting transition probabilities of this process thus satisfy
\[
p_k(\mathbf{t}_{k-1},\mathbf{t}_k)\propto\left\{\begin{array}{cc}
\phi(\text{deg}(w^*(\mathbf{t}_{k-1},\mathbf{t}_k), \mathbf{t}_{k-1})),&\mathbf{t}_{k-1}\subset\mathbf{t}_{k},\\
0,& \text{otherwise,}
\end{array}\right.\quad k\geq2.\]

Common examples of the function $\phi$ include (a) uniform attachment where $\phi(d) = 1$ for all $d \in \mathbb{N}$ (as in Example \ref{example:uniform}), (b) linear preferential attachment where $\phi(d) = d$ (as in Example \ref{example:linear-PA}), (c) uniform attachment on a $D$-regular tree where $\phi(d) = \max(0, D - d)$ for some $D \geq 2$, and (d) sublinear preferential attachment where $\phi(d) = d^\gamma$ for some $\gamma \in (0, 1)$. We also note that \cite{gao2017consistent} studies the estimation of the parameter function $\phi$. 
\end{example}

For labeled trees $\mathbf{t}, \mathbf{t}'$, not necessarily recursive, we write $\mathbf{t} \sim \mathbf{t}'$ if there exists a $\pi \in \text{Bi}(V(\mathbf{t}), V(\mathbf{t}'))$ such that $\pi \mathbf{t} = \mathbf{t}'$. We define the \emph{shape} of $\mathbf{t}$ as the equivalence class of all $\mathbf{t}'$ that are equivalent to $\mathbf{t}$ up to some relabeling, 
\[
\text{sh}(\mathbf{t}) := \{ \mathbf{t}' \text{ labeled tree} \,:\, \mathbf{t}' \sim \mathbf{t} \}.
\]
A shape is typically referred to as an ``unlabeled tree" or just a ``tree". We prefer the term ``shape" here to emphasize that $\text{sh}(\mathbf{t})$ does not refer to a single tree but rather to an equivalence class of all trees with a given structure.  Figure \ref{fig:illustration}(b) shows the `shape' of the labeled tree in Figure \ref{fig:illustration}(a).  In our analysis, this `shape' represents the set of all trees that produce the same structure after removing labels.


Since $\text{sh}(\mathbf{t}_n)$ is an equivalence class, we need to first define what it means to refer to the nodes of $\text{sh}(\mathbf{t}_n)$. To that end, we let $\mathcal{U}_n$ be an alphabet of $n$ distinct letters and represent the shape $\text{sh}(\mathbf{t}_n)$ by an arbitrary labeled tree $\mathbf{t}^*_n \in \text{sh}(\mathbf{t}_n)$ with node labels $\mathcal{U}_n$. We may now refer to the nodes of $\text{sh}(\mathbf{t}_n)$ through its alphabetically labeled representation $\mathbf{t}^*_n$. For convenience, we define $\tilde{\mathcal{T}}_n$ as the set of all labeled trees with $n$ nodes whose labels take values in the alphabet $\mathcal{U}_n$. 

We note that it is necessary to work with a labeled representation of a shape because (i) our inference questions make reference to properties of specific nodes (e.g., a particular node being the root) and (ii) computer programs require as input a labeled tree instead of an equivalent class of labeled trees. Because the choice of the labeled representation is arbitrary, we only consider inference methods that are independent of the choice of the representation. We formalize this with the notion of labeling-equivariance in Remark~\ref{rem:label-equivariance}.

\subsection{Inference Problem}
\label{sec:inference-problem}

Let $\mathbf{T}=(\mathbf{T}_k)_{k\geq1}$ be a Markovian growth process for which we observe only the unlabeled shape $\text{sh}(\mathbf{T}_n)$ for some $n\geq1$.
For a labeled representation $\mathbf{T}^*_n \in \text{sh}(\mathbf{T}_n)$ whose node labels take values in $\mathcal{U}_n$, we say that the \emph{root} of $\mathbf{T}_n$ is $v$ if the node labeled $v \in \mathcal{U}_n$ in the labeled representation $\mathbf{T}^*_n$ corresponds to the root of $\mathbf{T}_n$. More precisely, letting $\rho \in \text{Bi}([n], \mathcal{U}_n)$ be an unobserved bijection such that $\rho \mathbf{T}_n = \mathbf{T}^*_n$, we define
\begin{align}
\text{root}_\rho(\mathbf{T}_{n}) := \rho \mathbf{T}_1 = \rho(1) \in \mathcal{U}_n.
\label{eqn:root-node}
\end{align}

\begin{remark}
It is important to note that the root node depends on the choice of the isomorphism $\rho$; there could exist $\rho' \in \text{Bi}([n], \mathcal{U}_n)$ such that $\rho' \mathbf{T}_n = \mathbf{T}^*_n$ and that $\rho'(1) \neq \rho(1)$; see Figure~\ref{fig:root_isomorphism}. This is because multiple nodes of an unlabeled shape $\sh(\mathbf{T}_n)$ can be indistinguishable (see \eqref{eq:equivalent-nodes} in Section~\ref{sec:likelihood} for a formal definition of indistinguishable nodes). We show in  Remark~\ref{rem:label-equivariance} that this issue does not pose a problem to our inference procedure as long as we only consider \emph{labeling-equivariant} confidence sets. 
\end{remark}

\begin{figure}
    \centering
    \includegraphics[scale=.4]{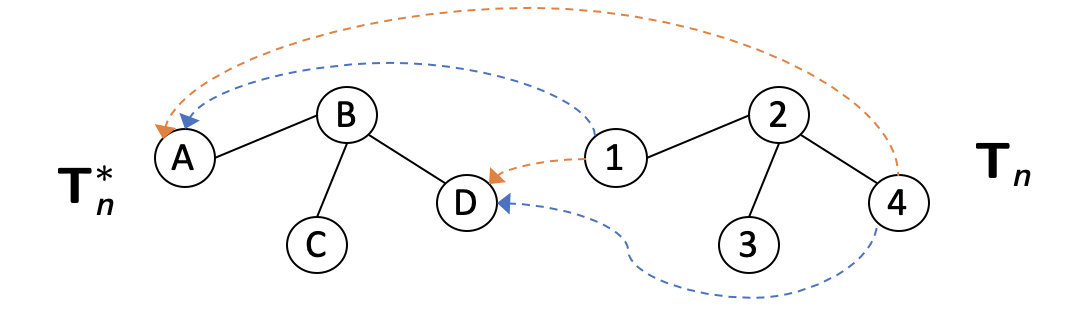}
    \caption{Let $\rho = (2 \mapsto B,\; 3 \mapsto C,\; 1 \mapsto A,\; 4 \mapsto D)$ (blue in figure) and let $\rho' =  (2 \mapsto B,\; 3 \mapsto C,\; 1 \mapsto D,\; 4 \mapsto A)$ (orange in figure).  Both $\rho$ and $\rho'$ are isomorphisms but we have that $\text{root}_\rho(\mathbf{T}_n) = A$ and $\text{root}_{\rho'}(\mathbf{T}_n) = D$.}
    \label{fig:root_isomorphism}
\end{figure}

Formally, the problem of root node inference is to construct, for given a confidence level $\epsilon \in (0,1)$, a confidence set $C_{\epsilon}(\mathbf{T}^*_n) \subset \mathcal{U}_n$ such that
\begin{align}
\mathbb{P}\bigl\{ \text{root}_\rho(\mathbf{T}_{n}) \in C_{\epsilon}( \mathbf{T}^*_n) ) \bigr\} \geq 1 - \epsilon. 
\label{eq:confidence-coverage}
\end{align}
As a trivial solution is to let $C_\epsilon$ be the set of all nodes $\mathcal{U}_n$, an important aspect of the inference problem is to make the confidence set $C_\epsilon(\cdot)$ as small as possible while still maintaining valid coverage~\eqref{eq:confidence-coverage}. We note that the root node cannot be consistently estimated since, for a tree of 2 nodes, it is impossible to distinguish which one is the root. 

\begin{remark}
\label{rem:label-equivariance}
Since our observation is the unlabeled network $\text{sh}(\mathbf{T}_n)$, it is natural to require the confidence set $C_{\epsilon}(\cdot)$ to be labeling-equivariant so it does not depend on the choice of the labeled representation $\mathbf{T}^*_n$. More precisely, for any $\tau \in \text{Bi}(\mathcal{U}_n, \mathcal{U}_n)$, we require
\begin{align}
\tau( C_{\epsilon}(\mathbf{T}^*_n)) = C_\epsilon(\tau \mathbf{T}^*_n), \label{eqn:label-equivariant}
\end{align}
where $\tau(C_\epsilon(\mathbf{T}^*_n))$ is the set containing the image of all members of $C_\epsilon(\mathbf{T}^*_n)$ under $\tau$.
In particular, if $\tau$ is an automorphism in the sense that $\tau \mathbf{T}^*_n = \mathbf{T}^*_n$, we require $C_\epsilon(\mathbf{T}^*_n)$ to be invariant with respect to $\tau$. 

The definition of the root node~\eqref{eqn:root-node} relies on a particular isomorphism $\rho$. However, 
for a labeling-equivariant confidence set, as in~\eqref{eqn:label-equivariant}, the probability of coverage~\eqref{eq:confidence-coverage} does not depend on this choice of the isomorphism. Indeed, if we write $u \in \mathcal{U}_n$ as the root node $\text{root}_\rho(\mathbf{T}_n)$ with respect to $\rho \in \text{Bi}([n], \mathcal{U}_n)$, then, for any $\tau \in \text{Bi}(\mathcal{U}_n, \mathcal{U}_n)$, the node labeled $\tau(u)$ is the root node under an alternative labeling $\rho \circ \tau$ and we see that $u = \text{root}_\rho(\mathbf{T}_n) \in C_\epsilon(\rho \mathbf{T}_n)$ if and only if $\tau(u) = \text{root}_{\tau \circ \rho}(\mathbf{T}_n) \in C_\epsilon((\tau \circ \rho) \mathbf{T}_n)$.

An unlabeled shape may contain indistinguishable nodes that are given different labels in a labeled representation. For example, the nodes labeled $A,C,D$ in the labeled tree in Figure~\ref{fig:root_isomorphism} are indistinguishable. However, any set of indistinguishable nodes must be either all included in or all excluded from a labeling-equivariant confidence set by the fact that such confidence sets are invariant with respect to automorphisms.
\end{remark}

Although we focus on root inference, the approach that we develop are applicable to more general inference problems such as inferring the arrival time of a particular node. We describe these in detail in Section~\ref{section:general}.

\subsection{Previous work on root inference}
\label{sec:prev-work}
For the problem of root node inference, \cite{bubeck2017finding} consider procedures that assign a centrality score to each node of the observed shape and then take the largest $K(\epsilon)$ nodes to be the confidence set where $K(\epsilon)$ is a size function whose value depends on the underlying distribution.  More precisely, let $\text{sh}(\mathbf{T}_n)$ be the observed shape with labeled representation $\mathbf{T}^*_n$ whose node labels take values in $\mathcal{U}_n$. Let $\psi \,:\, \mathcal{U}_n \times \tilde{\mathcal{T}}_n \rightarrow [0, \infty)$ be a scoring function and let $K(\epsilon)$ be an integer for any $\epsilon \in (0,1)$. Assuming that the nodes $u_1, u_2, \ldots, u_n \in \mathcal{U}_n$ are sorted so that $\psi(u_1, \mathbf{T}^*_n) \geq \psi(u_2, \mathbf{T}^*_n) \geq \ldots \geq \psi(u_n, \mathbf{T}^*_n)$, define $C_{K(\epsilon), \psi}(\mathbf{T}^*_n) := \{u_1, \ldots, u_{K(\epsilon)} \}$. The $\psi$ function is labeling-equivariant in that for any $\tau \in \text{Bi}(\mathcal{U}_n, \mathcal{U}_n)$, we have $\psi(u, \tilde{\mathbf{t}}_n) = \psi(\tau(u), \tau \tilde{\mathbf{t}}_n)$. The induced confidence set $C_{K(\epsilon), \psi}$ is therefore also labeling-equivariant. 

With these definitions, \citet[][Theorem~5]{bubeck2017finding} show that if the random recursive tree $\mathbf{T}_n$ has the uniform attachment distribution, then there exists a function $\psi$ such that, with any $K(\epsilon) \geq a \exp \biggl( \frac{b \log(1/\epsilon)}{\log \log (1/\epsilon)} \biggr)$ for universal constants $a,b > 0$, we have
\[
\liminf_{n \rightarrow \infty} \mathbb{P}\bigl( \text{root}_{\rho}(\mathbf{T}_n) \in C_{K(\epsilon), \psi}(\mathbf{T}^*_n) \bigr) \geq 1 - \epsilon.
\]
In other words, so long as $K(\epsilon)$ is large enough, $C_{K(\epsilon), \psi}(\mathbf{T}^*_n)$ has asymptotic confidence coverage of $1-\epsilon$. If $\mathbf{T}_n$ is distributed according to the preferential attachment distribution, then \citet[][Theorem~6]{bubeck2017finding} show that there exists $\psi$ such that $C_{K(\epsilon), \psi}$ has asymptotic coverage when $K(\epsilon) \geq C \frac{\log^2(1/\epsilon)}{\epsilon^4}$ for some universal constant $C > 0$. Lower bounds on the size $K(\epsilon)$ are also provided. We also note that after the completion of this manuscript, \citet[][Theorem~3.2]{banerjee2020root} further improved the upper bound on $K(\epsilon)$ to $\frac{C_1}{\epsilon^2} \exp \bigl( \sqrt{ C_2 \log \epsilon^{-1} } \bigr)$ in the linear preferential attachment setting.

If $\mathbf{T}_n$ is distributed according to uniform attachment on $D$-regular trees for some $D \geq 3$, then \citet[][Corollary~1]{khim2017confidence} shows that there exists $\psi$ such that $C_{K(\epsilon), \psi}$ has asymptotic coverage when $K(\epsilon) \geq C_D/\epsilon$ for some constant $C_D > 0$ depending only on $D$. 

These above results are surprising in that as the size of the tree increases, the size $K(\epsilon)$ of the confidence set can remain constant---the intuition for this being that the ``center" of the growing tree does not move significantly as new nodes arrive~\citep{jog2018persistence, jog2016analysis,  bubeck2015influence}. However, these results have a number of shortcomings that make them impractical for real applications. First, the confidence guarantee is asymptotic and does not hold in .  Second, the bound on the size $K(\epsilon)$ is theoretical and too conservative to be useful. As we show in Figures~\ref{tab:conf_size2} and~\ref{tab:conf_size3}, the bound $K(\epsilon)$ given in each case tends to be excessively conservative even for relatively large values of $n$.  Third, it is necessary to know the model of the random recursive tree $\mathbf{T}_n$ in order to choose the correct size $K(\epsilon)$. In the next section, we present an alternative approach to constructing confidence sets for the root node that addresses all these shortcomings.

The choice of the scoring function $\psi$ is crucial. The ideal choice is the likelihood, which is complicated to express so we defer its formal definition to equation~\eqref{eq:likelihood-definition} in Section~\ref{sec:likelihood} to avoid interrupting the flow of exposition. \cite{bubeck2017finding} remark that the likelihood is computationally intensive and analyzes a relaxation instead. One such relaxation is based on taking products of the sizes of the subtrees (termed by \cite{shah2011rumors} as \emph{rumour centrality}) and it plays an important role in our approach; see~\eqref{eq:conditional-root-probability} for a precise definition. Interestingly, one implication of our work is that this product-of-subtree-sizes relaxation in fact induces the same ordering of the nodes as the true likelihood so that the confidence set constructed from the relaxation is the same as that from the true likelihood. 

\section{Approach}\label{sec:approach}

In this section, we describe our approach to root inference through the notion of label randomization and shape exchangeability. In Section~\ref{section:general}, we show that the same approach applies to various other inference problems as well. 

\subsection{Random Labeling}

Let $\mathbf{T}_n$ be a random recursive tree and let $\text{sh}(\mathbf{T}_n)$ be its shape. Given $\mathbf{T}_n$ and any alphabetically labeled representation $\mathbf{T}^*_n \in\text{sh}(\mathbf{T}_n)$, we may independently generate a random bijection $\Lambda$ uniformly in $\text{Bi}(\mathcal{U}_n, \mathcal{U}_n)$ and apply it onto $\mathbf{T}^*_n$ to obtain a randomly labeled tree $\tilde{\mathbf{T}}_n := \Lambda\mathbf{T}^*_n$.  We note here that the resulting object satisfies $\tilde{\mathbf{T}}_n \stackrel{d}{=} \Pi\mathbf{T}_n$ where $\Pi\in\text{Bi}([n],\mathcal{U}_n)$ is another uniform random bijection chosen independently of $\mathbf{T}_n$ and $\Lambda$.  In particular, the marginal distribution of $\tilde{\mathbf{T}}_n$ does not depend on the choice of representative $\mathbf{T}^*_n \in\text{sh}(\mathbf{T}_n)$.  To fix notation, from now on we write $\tilde{\mathbf{T}}_n$ to denote a random labeled tree generated in this way.

\begin{figure}
\centering
\includegraphics[scale=.45]{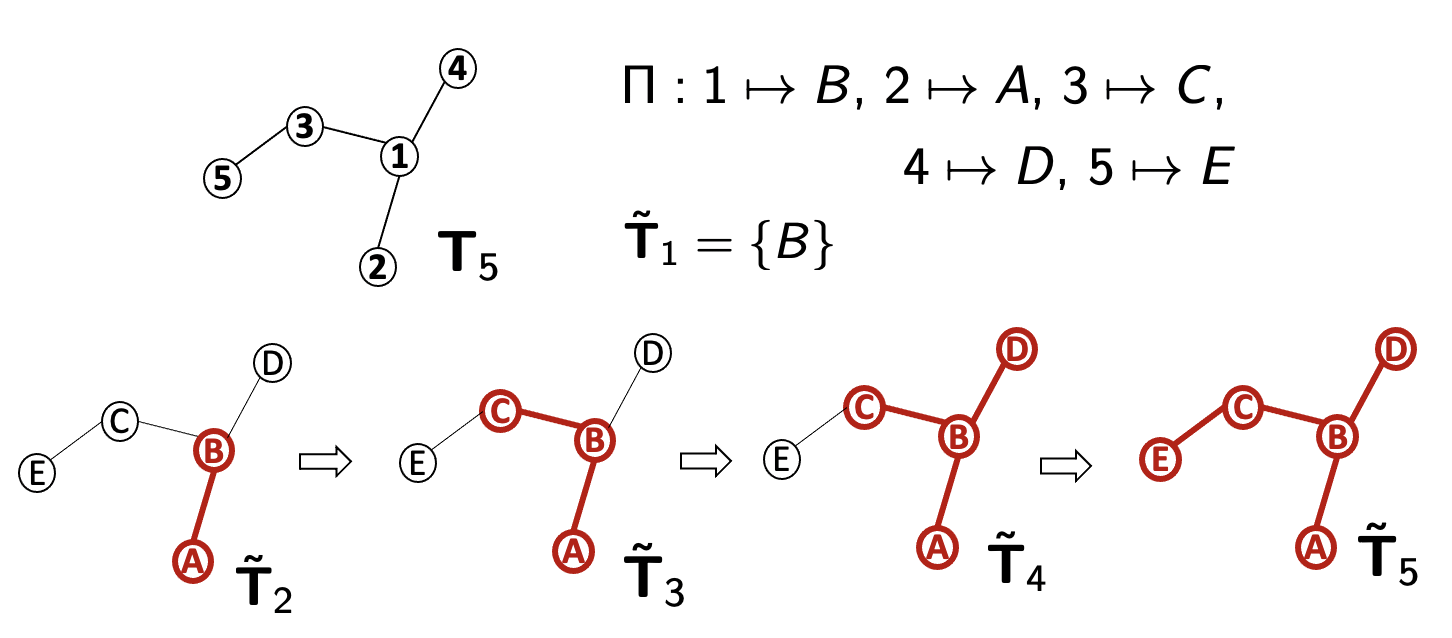}
    \caption{Illustration of label randomization with the corresponding history.}
    \label{fig:random_label}
\end{figure}


For any $k \in [n]$, we also define $\tilde{\mathbf{T}}_k :=
\Pi|_{[k]} \mathbf{T}_k$, where we interpret $\Pi|_{[k]}$ as the
domain restriction of the random bijection $\Pi$ defined on $[k]$.  In
this way, $\tilde{\mathbf{T}}_1 \subset \tilde{\mathbf{T}}_2 \subset
\ldots \subset \tilde{\mathbf{T}}_n$. Since each of the random
subtrees $\tilde{\mathbf{T}}_k$ has the same shape as $\mathbf{T}_k$,
the random shape $\text{sh}(\tilde{\mathbf{T}}_k)$ has the same
distribution as $\text{sh}(\mathbf{T}_k)$. In particular, the label of
the singleton $\tilde{\mathbf{T}}_1$ is drawn uniformly from
$\mathcal{U}_n$ and each subsequent node added at time $k$ is drawn
uniformly from $\mathcal{U}_n \backslash
V(\tilde{\mathbf{T}}_{k-1})$. We illustrate the definitions associated
with label randomization in Figure~\ref{fig:random_label}.

We can interpret label randomization as an augmentation of the probability space. An outcome for a Markov tree growth process is a recursive tree $\mathbf{t}_n \in \mathcal{T}_n^R$ whereas an outcome for the label randomized sequence $\tilde{\mathbf{T}}_1 \subset \tilde{\mathbf{T}}_2 \subset \ldots \subset \tilde{\mathbf{T}}_n$ is a pair $(\mathbf{t}_n, \pi)$ where $\mathbf{t}_n$ is a recursive tree and $\pi$ is a bijection from $[n]$ to $\mathcal{U}_n$. By defining $\tilde{\mathbf{t}}_k := \pi|_{[k]} \mathbf{t}_k$, we obtain a bijective correspondence between a pair $(\mathbf{t}_n, \pi)$ and a sequence of nested subtrees $\tilde{\mathbf{t}}_1 \subset \tilde{\mathbf{t}}_2 \subset \ldots \subset \tilde{\mathbf{t}}_n$ where $V(\tilde{\mathbf{t}}_n) = \mathcal{U}_n$. 
We define any such sequence of nested subtrees (equivalently any pair $(\mathbf{t}_n, \pi)$) as a \emph{tree history}, or simply a {\em history} for short. Intuitively, $\pi^{-1} \in \text{Bi}(\mathcal{U}_n, [n])$ is the isomorphism that gives the ordering information of the nodes in $\tilde{\mathbf{t}}_n$ as the node labels of $\tilde{\mathbf{t}}_n$ take value in the alphabet $\mathcal{U}_n$. Given the correspondence between the sequence $\tilde{\mathbf{t}}_1 \subset \tilde{\mathbf{t}}_2 \subset \ldots \subset \tilde{\mathbf{t}}_n$ and the pair $(\mathbf{t}_n,\pi)$, the probability of a history is given by
\begin{align}
    &\mathbb{P}( \tilde{\mathbf{T}}_1 = \tilde{\mathbf{t}}_1, \ldots, \tilde{\mathbf{T}}_{n-1} = \tilde{\mathbf{t}}_{n-1}, \tilde{\mathbf{T}}_n = \tilde{\mathbf{t}}_n) \nonumber \\
&= \mathbb{P}( \Pi \mathbf{T}_1 = \tilde{\mathbf{t}}_1, \ldots, \Pi \mathbf{T}_{n-1} = \tilde{\mathbf{t}}_{n-1}, \Pi \mathbf{T}_n = \tilde{\mathbf{t}}_n) \nonumber \\
&=\mathbb{P}( \mathbf{T}_1 = \pi^{-1} \tilde{\mathbf{t}}_1, \ldots, \mathbf{T}_{n-1} = \pi^{-1} \tilde{\mathbf{t}}_{n-1},  \mathbf{T}_n = \pi^{-1} \tilde{\mathbf{t}}_n \mid \Pi=\pi ) \mathbb{P}( \Pi = \pi) \nonumber \\
&=\mathbb{P}( \mathbf{T}_1 =  \mathbf{t}_1, \ldots, \mathbf{T}_{n-1} =  \mathbf{t}_{n-1},  \mathbf{T}_n = \mathbf{t}_n \mid \Pi_n=\pi ) \mathbb{P}( \Pi = \pi) \nonumber\\
&=\mathbb{P}( \mathbf{T}_1 =  \mathbf{t}_1, \ldots, \mathbf{T}_{n-1} =  \mathbf{t}_{n-1},  \mathbf{T}_n = \mathbf{t}_n  ) \mathbb{P}( \Pi = \pi) \label{eq:prob_single_history} 
\end{align}

To give a concrete example, let the random recursive tree $\mathbf{T}_n$ be generated from preferential attachment process $\text{PA}_\phi$. Then, the corresponding sequence of random trees $\tilde{\mathbf{T}}_1 \subset \tilde{\mathbf{T}}_2 \subset \ldots \subset \tilde{\mathbf{T}}_n$ have the distribution where $\tilde{\mathbf{T}}_1$ is a singleton node drawn uniformly from $\mathcal{U}_n$ and for $k \in \{2,3,\ldots,n\}$, we select a node $u \in V(\tilde{\mathbf{T}}_{k-1})$ with probability proportional to $\phi(\text{deg}(u, \tilde{\mathbf{T}}_{k-1}))$, select a node $v \in \mathcal{U}_n - V(\tilde{\mathbf{T}}_{k-1})$ uniformly at random, and add the edge $(u,v)$ to $\tilde{\mathbf{T}}_{k-1}$ to form $\tilde{\mathbf{T}}_k$. 

Suppose we have data in the form of a given unlabeled shape with an arbitrarily alphabetically labeled representation $\tilde{\mathbf{t}}_n$ whose node labels take values in $\mathcal{U}_n$. We may apply label randomization if necessary to assume without the loss of generality that $\tilde{\mathbf{t}}_n$ is an outcome of the randomly labeled tree $\tilde{\mathbf{T}}_n$. Our inference approach is based on the conditional distribution of a history
\begin{align}
\mathbb{P}(\tilde{\mathbf{T}}_1, \tilde{\mathbf{T}}_2, \ldots, \tilde{\mathbf{T}}_{n-1} \,|\, \tilde{\mathbf{T}}_n = \tilde{\mathbf{t}}_n). \label{eq:conditional-history-prob}
\end{align} 
For example, for $u \in \mathcal{U}_n$ and an observed shape $\mathbf{s}_n$ with alphabetically labeled representation $\tilde{\mathbf{t}}_n \in \mathbf{s}_n$, we may interpret $\mathbb{P}( \tilde{\mathbf{T}}_1 = \{u\} \,|\, \tilde{\mathbf{T}}_n = \tilde{\mathbf{t}}_n )$ as the probability of $u$ being the root node conditional on observing the shape $\mathbf{s}_n$. 

Label randomization is a data augmentation scheme that simplifies the analysis and the computation. We can define the conditional probability of a node being the root without the use of label randomization (see Remark~\ref{rem:unlabeled-conditional}) but we show in Theorem~\ref{thm:likelihood-equivalence} that the alternative definition is equivalent to the conditional root probability with label randomization. The calculation in \eqref{eq:prob_single_history} makes this approach precise.

Our proposed approach to either compute or approximate the conditional distribution of a history given a randomly labeled final state is especially natural for the purpose of developing a statistical framework which applies to a broad class of inference problems about the tree history.  This general strategy can also be found in approaches posed independently within other disciplines, as in the similar conditional probability-based approaches in the physics literature \citep{young2019phase, cantwell2019recovering, timar2020choosing} and the randomized labeling framework proposed for tree reconstruction algorithms in computer science   \citep{sreedharan2019inferring, magner2018times}.

\subsection{Root Inference}
\label{sec:root-inference}

With the definition of the conditional history~\eqref{eq:conditional-history-prob}, a natural approach for inferring the root is to construct a level $1-\epsilon$ \emph{credible} set by iteratively adding nodes with the largest conditional root probabilities until the sum of the conditional root probabilities among the nodes in the set exceed $1 - \epsilon$. More precisely, let our data be an unlabeled shape with an alphabetically labeled representation $\tilde{\mathbf{t}}_n$, that is, $\tilde{\mathbf{t}}_n$ has node labels in $\mathcal{U}_n$. We sort the nodes $u_1, \ldots, u_n \in \mathcal{U}_n$ of $\tilde{\mathbf{t}}_n$ such that 
\begin{align}
\mathbb{P}( \tilde{\mathbf{T}}_1 = \{u_1\} \,|\, \tilde{\mathbf{T}}_n = \tilde{\mathbf{t}}_n) \geq \mathbb{P}( \tilde{\mathbf{T}}_1 = \{u_2\} \,|\, \tilde{\mathbf{T}}_n = \tilde{\mathbf{t}}_n) \geq \ldots \geq \mathbb{P}( \tilde{\mathbf{T}}_1 = \{u_n\} \,|\, \tilde{\mathbf{T}}_n = \tilde{\mathbf{t}}_n), 
\label{eq:sorted_cond_root_probabilities}
\end{align}
and we define
\begin{align}
  K_{\epsilon}(\tilde{\mathbf{t}}_n) := \min \biggl \{ &k \in [n] \,:\, \sum_{i=1}^k \mathbb{P}( \tilde{\mathbf{T}}_1 = \{u_i\} \,|\, \tilde{\mathbf{T}}_n = \tilde{\mathbf{t}}_n) \geq 1 - \epsilon, \nonumber \\ 
  &\mathbb{P}(\tilde{\mathbf{T}}_1 = \{u_k\} \,|\, \tilde{\mathbf{T}}_n = \tilde{\mathbf{t}}_n) > \mathbb{P}(\tilde{\mathbf{T}}_1 = \{u_{k+1}\} \,|\, \tilde{\mathbf{T}}_n = \tilde{\mathbf{t}}_n) \biggr\}.
  \label{eq:size_K_definition}
\end{align}
We then define the $\epsilon$-credible set as
\begin{align}
 B_\epsilon(\tilde{\mathbf{t}}_n) := \{u_1, u_2, \ldots, u_{K_\epsilon(\tilde{\mathbf{t}}_n)}\}.
 \label{eq:B_set_definition}
 \end{align}

If there are no ties in the conditional root probabilities~\eqref{eq:sorted_cond_root_probabilities}, then $B_\epsilon(\tilde{\mathbf{t}}_n)$ is the smallest subset of $\mathcal{U}_n$ such that
\begin{align}
\mathbb{P}( \tilde{\mathbf{T}}_1 \in B_\epsilon(\tilde{\mathbf{t}}_n) \,|\, \tilde{\mathbf{T}}_n = \tilde{\mathbf{t}}_n ) \geq 1 - \epsilon,
\label{eq:conditional_coverage}
\end{align}
When there are ties however, the second condition in our definition of $K_\epsilon$ dictates that we resolve ties by including all nodes with equal conditional root probabilities. Breaking ties by inclusion ensures that $B_\epsilon(\tilde{\mathbf{t}}_n)$ is labeling-equivariant. 

In general, credible sets do not have valid Frequentist confidence coverage. However, our next theorem shows that in our setting, the credible set $B_\epsilon$ is in fact an honest confidence set. 

\begin{theorem}
\label{thm:frequentist-coverage}
Let $\mathbf{T}_n$ be a random recursive tree and let $\mathbf{T}^*_n$ be any arbitrary labeled representation (with labels taking values in $\mathcal{U}_n$) of the observed shape $\text{sh}(\mathbf{T}_n)$, and let $\rho \in \text{Bi}([n], \mathcal{U}_n)$ be any isomorphism such that $\rho \mathbf{T}_n = \mathbf{T}^*_n$. We have that, for any $\epsilon \in (0, 1)$,
\[
\mathbb{P}\bigl\{ \text{root}_\rho(\mathbf{T}_n) \in B_\epsilon(\mathbf{T}^*_n) \bigr\} \geq 1 - \epsilon.
\]
\end{theorem}

\begin{proof}
We first claim that, for a given shape with an alphabetically labeled representation $\tilde{\mathbf{t}}_n$, the credible set $B_\epsilon(\tilde{\mathbf{t}}_n)$ is labeling-equivariant (cf. Remark~\ref{rem:label-equivariance}) in the sense that for any $\tau \in \text{Bi}(\mathcal{U}_n, \mathcal{U}_n)$, we have that $\tau B_{\epsilon}(\tilde{\mathbf{t}}_n) = B_\epsilon(\tau \tilde{\mathbf{t}}_n)$.

Indeed, since $(\tilde{\mathbf{T}}_1, \tilde{\mathbf{T}}_2, \ldots, \tilde{\mathbf{T}}_n) \stackrel{d}{=} (\tau \tilde{\mathbf{T}}_1, \tau \tilde{\mathbf{T}}_2, \ldots, \tau \tilde{\mathbf{T}}_n)$, we have that, for any $u \in \mathcal{U}_n$,
\[
\mathbb{P}(\tilde{\mathbf{T}}_1 = \{u\} \,|\, \tilde{\mathbf{T}}_n = \tilde{\mathbf{t}}_n) = 
\mathbb{P}(\tilde{\mathbf{T}}_1 = \{\tau(u)\} \,|\,
\tilde{\mathbf{T}}_n = \tau\tilde{\mathbf{t}}_n).
\]
Therefore, for any $u, v \in \mathcal{U}_n$, we have that $\mathbf{P}( \tilde{\mathbf{T}}_1 = \{u\} \,|\, \tilde{\mathbf{T}}_n = \tilde{\mathbf{t}}_n) \geq \mathbf{P}( \tilde{\mathbf{T}}_1 = \{v\} \,|\, \tilde{\mathbf{T}}_n = \tilde{\mathbf{t}}_n)$ if and only if 
$\mathbf{P}( \tilde{\mathbf{T}}_1 = \{\tau(u)\} \,|\, \tilde{\mathbf{T}}_n = \tau \tilde{\mathbf{t}}_n) \geq \mathbf{P}( \tilde{\mathbf{T}}_1 = \{\tau(v)\} \,|\, \tilde{\mathbf{T}}_n = \tau \tilde{\mathbf{t}}_n)$.
Since $B_\epsilon(\mathbf{T}^*_n)$ is constructed by taking the top elements of $\mathcal{U}_n$ that maximizes the cumulative conditional root probabilities, the claim follows. 

Now, let $\rho \in \text{Bi}([n], \mathcal{U}_n)$ be such that $\rho \mathbf{T}_n = \mathbf{T}^*_n$ and let $\Pi$ be a random bijection drawn uniformly in $\text{Bi}(\mathcal{U}_n, \mathcal{U}_n)$. Then,
\begin{align*}
\mathbb{P}( \text{root}_\rho(\mathbf{T}_n) \in B_\epsilon(\mathbf{T}^*_n)) &=
\mathbb{P}( \rho(1) \in B_\epsilon(\rho \mathbf{T}_n) ) \\
&= \mathbb{P}\bigl\{ (\Pi \circ \rho)(1) \in B_\epsilon( (\Pi \circ \rho) \mathbf{T}_n) \,|\, \Pi = \text{Id} \bigr\}\\
&= \mathbb{P}\bigl\{ (\Pi \circ \rho)(1) \in B_\epsilon( (\Pi \circ \rho) \mathbf{T}_n) \bigr\} \\
&= \mathbb{P}( \tilde{\mathbf{T}}_1 \in B_\epsilon(\tilde{\mathbf{T}}_n)) \geq 1 - \epsilon,
\end{align*}
where the penultimate equality follows from the labeling-equivariance of $B_\epsilon$ and where the last inequality follows because $\mathbf{P}( \tilde{\mathbf{T}}_1 \in B_\epsilon(\tilde{\mathbf{T}}_n) \,|\, \tilde{\mathbf{T}}_n = \tilde{\mathbf{t}}_n) \geq 1 - \epsilon$ for any labeled tree $\tilde{\mathbf{t}}_n$ (with labels in $\mathcal{U}_n$) by the definition of $B_{\epsilon}$.
\end{proof}

Theorem~\ref{thm:frequentist-coverage} shows that we may obtain a valid confidence set by constructing a credible set. The credible set can be efficiently computed for a class of tree growth processes that we describe in the next section. 

\subsection{Shape exchangeable growth process}

The conditional history distribution~\eqref{eq:conditional-history-prob} can be intractable for a general Markov tree growth process but it has an elegant characterization when the growth process satisfies a \emph{shape exchangeability} condition.

\begin{definition}
\label{defn:shape-exch}
A random recursive tree process  $\mathbf{T}=(\mathbf{T}_n)_{n\geq0}$ is {\em shape exchangeable} if for all $n\geq1$
\begin{equation}\label{eq:shape-exch}
\Pb(\Treen=\tree_n)=\Pb(\Treen=\tree_n')\quad\text{for all recursive trees }\tree_n,\tree_n' \in \mathcal{T}^R_n \text{ satisfying }\sh(\tree_n)=\sh(\tree_n').
\end{equation}
\end{definition}

\begin{remark}
Our definition of shape exchangeable trees follows recent developments in the theory of exchangeability, in particular the theory of relative exchangeability \citep{CraneTowsner2018}, in which the distribution is invariant with respect to a subgroup of permutations that respect an underlying structure.  In the case of shape exchangeability, the shape determines an isomorphism class of tree histories, and shape exchangeability implies that every member of a given isomorphism class is equiprobable.  Note, in particular, that the exchangeability condition applies to the distribution on recursive trees (i.e., tree histories), in the sense that two histories with the same shape have the same probability.  Shape exchangeability does not imply that each shape has the same probability of occurring.
\end{remark}

Immediate examples of shape exchangeable processes include the uniform attachment and the linear preferential attachment models from Examples \ref{example:uniform} and \ref{example:linear-PA}.  Theorem \ref{thm:pa-shape-exchangeable} shows that these two classes combine to characterize the class of all shape exchangeable processes.

In general, if a random recursive tree $\mathbf{T}_n$ is shape exchangeable, then the conditional probability~\eqref{eq:conditional-history-prob} of the random sequence of label-randomized trees $\tilde{\mathbf{T}}_1, \ldots, \tilde{\mathbf{T}}_n$ takes on a simple form as shown in Proposition~\ref{prop:uniform-conditional} below. We define some necessary concepts and notation before stating the result.

Let $\tilde{\mathbf{t}}_n$ be a labeled tree with nodes labeled by $\mathcal{U}_n$. We define $\text{hist}(\tilde{\mathbf{t}}_n)$ as set of all histories $\tilde{\mathbf{t}}_1 \subset \tilde{\mathbf{t}}_2 \subset \ldots \subset \tilde{\mathbf{t}}_{n-1} \subset \tilde{\mathbf{t}}_n$ that result in $\tilde{\mathbf{t}}_n$. We may associate each distinct history  $\tilde{\mathbf{t}}_1 \subset \tilde{\mathbf{t}}_2 \subset \ldots \subset \tilde{\mathbf{t}}_n$ with a sequence $v_1, v_2, \ldots, v_n \in \mathcal{U}_n$ such that $\tilde{\mathbf{t}}_k = \tilde{\mathbf{t}}_n \cap \{v_1, \ldots, v_k\}$ for every $k \in [n]$. We thus have that
\[
\text{hist}(\tilde{\mathbf{t}}_n) = \{ v_1, \ldots, v_n \in \mathcal{U}_n \,:\, \forall k \in [n], \, \tilde{\mathbf{t}}_n \cap \{v_1, \ldots, v_k\} \text{ is connected}
\},
\]
that is, $\text{hist}(\tilde{\mathbf{t}}_n)$ is a set of the permutations $v_1, v_2, \ldots, v_n$ of the node label set $\mathcal{U}_n$ that satisfy the constraint that $\tilde{\mathbf{t}}_n$ restricted to $v_1, \ldots, v_k$ is a connected sub-tree for any $k \in [n]$. Equivalently, we may define $\text{hist}(\tilde{\mathbf{t}}_n)$ as the set of all label bijections $\pi \in \text{Bi}([n], \mathcal{U}_n)$ such that $\pi^{-1} \tilde{\mathbf{t}}_n$ is a recursive tree. 

We write $\# \text{hist}(\tilde{\mathbf{t}}_n)$ the denote the number of distinct histories of $\tilde{\mathbf{t}}_n$. It is clear that $\# \text{hist}(\tilde{\mathbf{t}}_n)$ depends only on the shape $\text{sh}(\tilde{\mathbf{t}}_n)$. Moreover, for a particular node $v \in \mathcal{U}_n$, we define $\text{hist}(\tilde{\mathbf{t}}_n, v) $ as the set of all histories rooted at the node $v$, that is,
\[
\text{hist}(\tilde{\mathbf{t}}_n, v) = \{ v_1, \ldots, v_n \in \mathcal{U}_n \,:\, v_1 = v, \,\, \forall k \in [n], \, \tilde{\mathbf{t}}_n \cap \{v_1, \ldots, v_k\} \text{ is connected}
\}.
\]

\begin{proposition}
\label{prop:uniform-conditional}
Suppose $\mathbf{T}_n$ is shape exchangeable and let $\tilde{\mathbf{T}}_{1:n} = (\tilde{\mathbf{T}}_1, \ldots, \tilde{\mathbf{T}}_n)$ be the randomly labeled history. Then, for any history $\tilde{\mathbf{t}}_1 \subset \tilde{\mathbf{t}}_2 \subset \ldots \subset \tilde{\mathbf{t}}_n$ with labels in $\mathcal{U}_n$, we have that
\[
\mathbb{P}( \tilde{\mathbf{T}}_1 = \tilde{\mathbf{t}}_1, \ldots, \tilde{\mathbf{T}}_{n-1} = \tilde{\mathbf{t}}_{n-1} \,|\, \tilde{\mathbf{T}}_n = \tilde{\mathbf{t}}_n ) = \frac{1}{\# \text{hist}(\tilde{\mathbf{t}}_n) }
\]
\end{proposition}

\begin{proof}
Suppose $\mathbf{T}_n$ is shape exchangeable. Let  $\tilde{\mathbf{t}}_n$ be any labeled tree with nodes labels in $\mathcal{U}_n$ let  $\mathbf{t}^{\bullet}_1\subset\cdots\subset\mathbf{t}^{\bullet}_{n-1} \subset \tilde{\mathbf{t}}_n$ and $\mathbf{t}^{\circ}_1\subset\cdots\subset \mathbf{t}^\circ_{n-1} \subset \tilde{\mathbf{t}}_n$ 
be two histories of $\tilde{\mathbf{t}}_n$. By~\eqref{eq:prob_single_history} and shape exchangeability of $\mathbf{T}_n$, we have that
\[
\mathbb{P}(\tilde{\mathbf{T}}_1 = {\mathbf{t}}^{\bullet}_1, \ldots, \tilde{\mathbf{T}}_{n-1} = {\mathbf{t}}^{\bullet}_{n-1}, \tilde{\mathbf{T}}_n = \tilde{\mathbf{t}}_n) = 
\mathbb{P}(\tilde{\mathbf{T}}_1 = {\mathbf{t}}^{\circ}_1, \ldots, \tilde{\mathbf{T}}_{n-1} = {\mathbf{t}}^{\circ}_{n-1}, \tilde{\mathbf{T}}_n = \tilde{\mathbf{t}}_n) .
\]

Therefore,
\begin{align*}
&\mathbb{P}( \tilde{\mathbf{T}}_1 = {\mathbf{t}}^{\bullet}_1, \ldots, \tilde{\mathbf{T}}_{n-1} = {\mathbf{t}}^{\bullet}_{n-1} \,|\, \tilde{\mathbf{T}}_n = \tilde{\mathbf{t}}_n) \\
&= \mathbb{P}(\tilde{\mathbf{T}}_1 = {\mathbf{t}}^{\bullet}_1, \ldots, \tilde{\mathbf{T}}_{n-1} = {\mathbf{t}}^{\bullet}_{n-1}, \tilde{\mathbf{T}}_n = \tilde{\mathbf{t}}_n)\frac{1}{\mathbb{P}(\tilde{\mathbf{T}}_n=\tilde{\mathbf{t}}_n)}\\
&= \mathbb{P}(\tilde{\mathbf{T}}_1 = {\mathbf{t}}^{\circ}_1, \ldots, \tilde{\mathbf{T}}_{n-1} = {\mathbf{t}}^{\circ}_{n-1}, \tilde{\mathbf{T}}_n = \tilde{\mathbf{t}}_n)\frac{1}{\mathbb{P}(\tilde{\mathbf{T}}_n=\tilde{\mathbf{t}}_n)}\\
&=\mathbb{P}( \tilde{\mathbf{T}}_1 = {\mathbf{t}}^{\circ}_1, \ldots, \tilde{\mathbf{T}}_{n-1} = {\mathbf{t}}^{\circ}_{n-1} \,|\, \tilde{\mathbf{T}}_n = \tilde{\mathbf{t}}_n),
\end{align*}
for all histories $(\mathbf{t}^{\bullet}_i)_{1\leq i \leq n}$ and $(\mathbf{t}^{\circ}_i)_{1\leq i \leq n}$ corresponding to $\tilde{\mathbf{t}}_n$. It follows that the conditional probability is equal for all histories, and thus the conditional distribution of the history given the final state $\tilde{\mathbf{t}}_n$ is uniform over $\text{hist}(\tilde{\mathbf{t}}_n)$, as was to be proven.
\end{proof}


\begin{theorem}
\label{thm:pa-shape-exchangeable}
Let $\phi \,:\, \mathbb{N} \rightarrow [0,\infty)$ and let $\mathbf{T}=(\mathbf{T}_n)_{n\geq1}$ be a $\text{PA}_\phi$ process. $\mathbf{T}$ is shape exchangeable if and only if $\phi$ has the form
\begin{equation}\label{eq:pa-shape-exchangeable-phi}
\phi(d) = \max(\alpha+\beta d,0),\quad d\geq1,
\end{equation}
for $\alpha,\beta$ satisfying either
\begin{itemize}
    \item $\beta<0$ and $\alpha=-D\beta$ for some integer $D\geq2$ or
    \item $\beta\geq0$ and $\alpha>-\beta$.
\end{itemize}
\end{theorem}

We defer the proof of Theorem~\ref{thm:pa-shape-exchangeable} to
Section~\ref{sec:proof-pa-shape} of the appendix.

\begin{remark}
Theorem \ref{thm:pa-shape-exchangeable} shows that shape exchangeability encompasses three widely studied tree growth processes.  Both the linear preferential attachment process, where $\phi(d) = d$, and the uniform attachment, where $\phi(d) = 1$, are shape exchangeable.  In addition, we note that the case where $\beta$ is negative corresponds to uniform attachment on the $D$-regular tree, i.e., $\phi(d)=D-d$ for some $D\geq2$.  However, the sublinear preferential attachment process where $\phi(d) = d^\gamma$ for some $\gamma \in (0, 1)$ is not shape exchangeable; we discuss inference procedures for non-shape exchangeable trees in Section~\ref{sec:importance_sampling}. 


Thus, for uniform attachment, linear preferential attachment, and $D$-regular uniform attachment processes, Proposition \ref{prop:uniform-conditional} and Theorem \ref{thm:pa-shape-exchangeable} combine to imply that inference about measurable functions of the unobserved history can be performed without knowing the $\alpha$ and $\beta$ parameters governing the process. In particular, valid confidence sets can be constructed by observing only the shape of the final state. 

Equivalently, if $\mathbf{T}_n$ has the $\text{PA}_\phi$ distribution where $\phi(d) = \max(0, \alpha + \beta d)$, then, by Proposition~\ref{prop:uniform-conditional}, the shape $\text{sh}(\mathbf{T}_n)$ is a sufficient statistic for $\alpha$ and $\beta$ and knowledge of the history is ancillary to the estimation of $\alpha$ and $\beta$. We note that \cite{gao2017consistent} makes a similar informal statement regarding the estimation of the parameter function $\phi$. 
\end{remark}

\subsection{Computing conditional root probabilities}

Shape exchangeable tree processes are naturally suited to the inference questions highlighted in Section~\ref{sec:inference-problem}.  For example, for the question of root inference, suppose we observe $\sh(\Treen)=\shape$ with an arbitrary labeled representation $\tilde{\mathbf{t}}_n$, then we may compute, for each node $v \in \mathcal{U}_n$ of $\tilde{\mathbf{t}}_n$,
\begin{align}
\mathbb{P}( \tilde{\mathbf{T}}_1 = \{v\} \,|\, \tilde{\mathbf{T}}_n = \tilde{\mathbf{t}}_n ) =
\frac{\#\text{hist}(\tilde{\mathbf{t}}_n, v)}{\# \text{hist}(\tilde{\mathbf{t}}_n)}. \label{eq:conditional-root-probability}
\end{align}

In fact, the numerator $\# \text{hist}(\tilde{\mathbf{t}}_n, u)$ coincides with the notion of \emph{rumor centrality} defined by \citet{shah2011rumors} for the purpose of estimating the root node. The following proposition summarizes a characterization of $\#\text{hist}(\tilde{\mathbf{t}}_n,u)$ given in Section IIIA of \cite{shah2011rumors}. We note that \cite{knuth1997art} made the same observation in the context of counting the number of ways to linearize a partial ordering. 

\begin{proposition}
\citep{knuth1997art, shah2011rumors} \\
Let $\tilde{\mathbf{t}}_n$ be a labeled tree with $V(\tilde{\mathbf{t}}_n) = \mathcal{U}_n$ and let $u \in \mathcal{U}_n$. Viewing $\tilde{\mathbf{t}}_n$ as being rooted at $u$, define, for every node $v \in \mathcal{U}_n$, the tree $\tilde{\mathbf{t}}^{(u)}_v$ as the sub-tree of $\tilde{\mathbf{t}}_n$ rooted at $v$ and $n^{(u)}_v$ as the size of $\tilde{\mathbf{t}}^{(u)}_v$. Writing $v_1, v_2, \ldots, v_L$ as the neighbors of $u$, we have that
\begin{align*}
\#\text{hist}(\tilde{\mathbf{t}}_n, u) &= \biggl( \frac{(n-1)!}{n^{(u)}_{v_1}! n^{(u)}_{v_2}! \ldots n^{(u)}_{v_L}!} \# \text{hist}(\tilde{\mathbf{t}}^{(u)}_{v_1}, v_1)\# \text{hist}(\tilde{\mathbf{t}}^{(u)}_{v_2}, v_2) \ldots \# \text{hist}(\tilde{\mathbf{t}}^{(u)}_{v_L}, v_L) \biggr) \\
&= (n-1)! \prod_{v \in \mathcal{U}_n - \{u\}} \frac{1}{ n^{(u)}_v }.
\end{align*}
\label{prop:count-rooted-history}
\end{proposition}

Using the fact that $\#\text{hist}(\tilde{\mathbf{t}}_n, v) = \#\text{hist}(\tilde{\mathbf{t}}_n, \text{pa}(v)) \frac{n^{(u)}_{v}}{n - n^{(u)}_{v}}$ for any node $v$ and its parent node $\text{pa}(v)$, viewing $\tilde{\mathbf{t}}_n$ as being rooted at $u$, \citet{shah2011rumors} derive an $O(n)$ algorithm  for counting the number of histories for all possible roots $\{\#\text{hist}(\tilde{\mathbf{t}}_n, u)\}_{u \in \mathcal{U}_n}$. We give the details in Algorithm~\ref{alg:count_history} for reader's convenience. Using Algorithm~\ref{alg:count_history}, we conclude that the overall runtime of computing the confidence set $B_\epsilon(\cdot)$ is $O(n \log n)$ since we need to also peform a sort. 

\begin{algorithm}
\caption{Computing $\{ \#\text{hist}(\tilde{\mathbf{t}}_n, v)\}_{v \in \mathcal{U}_n}$ \citep{shah2011rumors}}
\label{alg:count_history}
\textbf{Input:} a labeled tree $\tilde{\mathbf{t}}_n$. \\
\textbf{Output:} $\# \text{hist}(\tilde{\mathbf{t}}_n, v)$ for all nodes $v \in \mathcal{U}_n$.
\begin{algorithmic}
\State Arbitrarily select root $u \in \mathcal{U}_n$.
\For{$v \in \mathcal{U}_n$}
 \State Compute and store $n^{(u)}_v := | \tilde{\mathbf{t}}^{(u)}_v |$.
 \EndFor
\State Compute $\#\text{hist}(\tilde{\mathbf{t}}_n, u)$ by Proposition~\ref{prop:count-rooted-history} and set $\mathcal{S} = \{ \text{Children}(u) \}$.
\While{$\mathcal{S}$ is not empty}
 \State Remove $v \in \mathcal{S}$
 \State Compute $\#\text{hist}(\tilde{\mathbf{t}}_n, v) = \#\text{hist}(\tilde{\mathbf{t}}_n, \text{pa}(v)) \frac{ n_v^{(u)}}{n- n_v^{(u)}}$
 \State Add $\text{Children}(v)$ to $\mathcal{S}$
 \EndWhile
\end{algorithmic}
\end{algorithm}

\subsection{Size of the confidence set}

In this section, we use the results of \cite{bubeck2017finding} and \cite{khim2017confidence} to provide a theoretical analysis of the size of the confidence set $B_\epsilon(\cdot)$. We also provide empirical studies of the size in our simulation studies in Section~\ref{sec:simulation}. 

In Section~\ref{sec:root-inference}, we defined $K_\epsilon(\tilde{\mathbf{t}}_n)$ for a fixed labeled tree $\tilde{\mathbf{t}}_n$ in~\eqref{eq:size_K_definition} which is the size of our confidence set $B_\epsilon(\tilde{\mathbf{t}}_n)$. In this section, we analyze a slight variation where, assuming that the nodes $u_1, u_2, \ldots, u_n$ are sorted in decreasing order by their conditional root probabilities, we define
\begin{align}
K_\epsilon(\tilde{\mathbf{t}}_n) := \min \biggl\{ k \in [n] \,:\,  \sum_{i=1}^k \mathbb{P}(\tilde{\mathbf{T}}_1 = \{u_i\} \,|\, \tilde{\mathbf{T}}_n = \tilde{\mathbf{t}}_n ) \geq 1 - \epsilon 
\nonumber \\
\,:\, \{u_1, \ldots, u_k\} = \bigcup_{i = 1}^k \text{Eq}(u_i, \tilde{\mathbf{t}}_n) \biggr\} \label{eq:size-K-definition2}
\end{align}
where $\text{Eq}(u, \tilde{\mathbf{t}}_n) \subset \mathcal{U}_n$ is defined formally in~\eqref{eq:equivalent-nodes} and is intuitively the set of nodes equivalent to $u$ in the tree $\tilde{\mathbf{t}}_n$. If we construct $B_\epsilon(\tilde{\mathbf{t}}_n) := \{ u_1, \ldots, u_{K_\epsilon(\tilde{\mathbf{t}}_n)}\}$, then, for any fixed labeled tree $\tilde{\mathbf{t}}_n$, the set $B_\epsilon(\tilde{\mathbf{t}}_n)$ is the smallest labeling-equivariant subset of $\mathcal{U}_n$ such that $\mathbb{P}(\tilde{\mathbf{T}}_1 \in B_\epsilon(\tilde{\mathbf{t}}_n) \,|\, \tilde{\mathbf{T}}_n = \tilde{\mathbf{t}}_n) \geq 1 - \epsilon$. In other words, $B_\epsilon(\cdot)$ is the optimal credible set for any fixed Bayesian coverage level. We may directly apply the argument in the proof of Theorem~\ref{thm:frequentist-coverage} to show that $B_\epsilon(\cdot)$, defined with respect to~\eqref{eq:size-K-definition2} is also a valid Frequentist confidence set at the same level $1-\epsilon$. We note that, for practical applications, we prefer~\eqref{eq:size_K_definition} over~\eqref{eq:size-K-definition2}; we observe in simulations that the two are almost always equivalent. The next lemma compares the size of $B_\epsilon(\cdot)$ with the optimal Frequentist confidence set.

\begin{lemma}
\label{lem:comparison}
Let $\epsilon \in (0, 1)$ be arbitrary, let $\mathbf{T}_n$ be a random recursive tree, and let $\mathbf{T}^*_n$ be an alphabetically labeled representation of $\text{sh}(\mathbf{T}_n)$. Let $K_\epsilon(\mathbf{T}^*_n)$ be defined as in~\eqref{eq:size-K-definition2}. Fix any $\delta \in (0,1)$ and let $C_{\delta \epsilon}(\mathbf{T}_n^*)$ be any confidence set for the root node that is labeling-equivariant and has asymptotic coverage level $\delta \epsilon$, that is, $\limsup_{n \rightarrow \infty} \mathbb{P}( \text{root}_{\rho}(\mathbf{T}_n) \notin C_{\delta \epsilon}( \rho \mathbf{T}_n)) \leq \delta \epsilon$ for any labeling $\rho \in \text{Bi}([n], \mathcal{U}_n)$ that satisfies $\rho \mathbf{T}_n = \mathbf{T}_n^*$. Then, we have that
\[
\limsup_{n \rightarrow \infty} \mathbb{P}\bigl( K_{\epsilon}(\mathbf{T}_n^*) \geq \# C_{\delta \epsilon}(\mathbf{T}_n^*) \bigr) \leq \delta.
\]
\end{lemma}

By letting $\delta = 1/2$ for example, Lemma~\ref{lem:comparison} shows that, with probability at least $1/2$, the size of our confidence set $B_\epsilon(\cdot)$ is no larger than the optimal asymptotically valid confidence set at a higher level $1 - \epsilon/2$. We defer the proof of Lemma~\ref{lem:comparison} to Section~\ref{sec:size-proof} in the appendix.

We now define $K_{
\text{ua}}(\epsilon) := a \exp\biggl( \frac{b \log(1/\epsilon)}{\log \log (1/\epsilon)} \biggr)$, $K_{\text{lpa}}(\epsilon) := C \frac{\log^2 (1/\epsilon)}{\epsilon^4}$, and, for an integer $D \geq 3$, 
$K_{\text{reg}, D}(\epsilon) := C_D/\epsilon$ where $a, b, C >0$ are universal constants and where $C_D > 0$ is a constant that depend only on $D$. We may then use Lemma~\ref{lem:comparison} in conjunction with results from \cite{bubeck2017finding} and \cite{khim2017confidence} to bound the size of our confidence sets. 

\begin{corollary}
Let $\epsilon \in (0, 1)$ be arbitrary, let $\mathbf{T}_n$ be a random recursive tree, and let $\mathbf{T}^*_n$ be an arbitrary labeled representation of $\text{sh}(\mathbf{T}_n)$ whose node labels take values in $\mathcal{U}_n$. Let $K_\epsilon(\mathbf{T}^*_n)$ be defined as in~\eqref{eq:size-K-definition2}.

If $\mathbf{T}_n$ is distributed according to the uniform attachment model, we have that, for all $\delta \in (0, 1)$,
\[
\limsup_{n \rightarrow \infty} \mathbb{P}( K_{\epsilon}(\mathbf{T}^*_n) \geq K_{\text{ua}}(\delta \cdot \epsilon)) \leq \delta.
\]
If $\mathbf{T}_n$ is distributed according to linear preferential attachment,
\[
\limsup_{n \rightarrow \infty} \mathbb{P}( K_{\epsilon}(\mathbf{T}^*_n) \geq K_{\text{pa}}(\delta \cdot \epsilon) ) \leq \delta,
\]
If, for some integer $D \geq 3$, $\mathbf{T}_n$ is distributed according to uniform attachment on a $D$-regular tree, then
\[
\limsup_{n \rightarrow \infty} \mathbb{P}( K_{\epsilon}(\mathbf{T}^*_n) \geq K_{\text{reg}, D}(\delta \cdot \epsilon) ) \leq \delta.
\]
\label{cor:confidence_set_size}
\end{corollary}

\begin{proof}

By \citet[][Theorem~5]{bubeck2017finding}, we know that when $\mathbf{T}_n$ has the uniform attachment distribution, there exists a labeling-equivariant scoring function $\psi$ such that for any $\epsilon \in (0,1)$, the set $C_{K_{ua}(\epsilon), \psi}(\cdot)$ with size $K_{ua}(\epsilon)$ contains the root with at least $1-\epsilon$ probability asymptotically. The first part of the Theorem thus follows.

We may obtain the other two claims of the Theorem in identical ways by using \citet[][Theorem~6]{bubeck2017finding} and \citet[][Corollary~1]{khim2017confidence}.
\end{proof}

From Corollary~\ref{cor:confidence_set_size}, we see that in all three cases, the random size $K_\epsilon(\mathbf{T}^*_n) = \#B_{\epsilon}(\mathbf{T}^*_n)$ is $O_p(1)$ as $n \rightarrow \infty$, which shows that the size of the confidence set is of a constant order even when the number of nodes tends to infinity. Moreover, the median size is asymptotically at most $K_{ua}(\epsilon/2), K_{pa}(\epsilon/2), K_{\text{reg}, D}(\epsilon/2)$ respectively for each of the three cases. As we show in our simulation studies (see e.g. Tables~\ref{tab:conf_size1}), these bounds tend to be very conservative. 

Since the size of our confidence set $K_\epsilon(\mathbf{T}^*_n)$ depends on the observed tree $\mathbf{T}^*_n$, it is adaptive to the underlying distribution of $\mathbf{T}_n$. In contrast, the sizes of the confidence sets considered in \cite{bubeck2017finding} and \cite{khim2017confidence} depend only on $\epsilon$ and hence must be chosen with knowledge of the true model. 

\subsection{Equivalence to maximum likelihood}
\label{sec:likelihood}

In this section, we show that $\mathbb{P}(\tilde{\mathbf{T}}_1 = \{u \} \,|\, \tilde{\mathbf{T}}_n = \tilde{\mathbf{t}}_n)$ is proportional to the likelihood of $u$ being the root on observing $\text{sh}(\mathbf{T}_n) = \text{sh}(\tilde{\mathbf{t}}_n)$. Thus, any confidence sets created by ordering the nodes according to their conditional root probability $\mathbb{P}(\tilde{\mathbf{T}} = \{u \} \,|\, \tilde{\mathbf{T}}_n = \tilde{\mathbf{t}}_n)$ also maximizes the likelihood. We first follow \citet[][Section~3]{bubeck2017finding} to derive the likelihood of a node $u$ being the root on observing the unlabeled shape $\text{sh}(\mathbf{T}_n)$. 

Given any labeled trees $\mathbf{t}, \mathbf{t}'$, not necessarily recursive, and two nodes $u \in V(\mathbf{t},), u' \in V(\mathbf{t}')$, we say that $(\mathbf{t}, u)$ and $(\mathbf{t}', u')$ have equivalent rooted shape (written $(\mathbf{t}, u) \sim_0 (\mathbf{t}', u)$) if there exists an isomorphism $\tau \in \text{Bi}(V(\mathbf{t}), V(\mathbf{t}'))$ such that $\tau \mathbf{t} = \mathbf{t}'$ and $\tau(u) = u'$. We then define the rooted shape of $(\mathbf{t}, u)$ as the equivalence class
\[
\text{sh}_0(\mathbf{t}, u) = \{ \mathbf{t}' \text{ labeled tree, } u' \in V(\mathbf{t}')  \,:\, (\mathbf{t}', u') \sim_0 (\mathbf{t}, u) \}.
\]
We give examples of rooted shapes with 4 nodes in Figure~\ref{fig:4nodes-example}.

For any labeled tree $\mathbf{t}$, not necessarily recursive, and for any node $u \in V(\mathbf{t})$, we define the set of indistinguishable nodes as 
\begin{align}
\text{Eq}(u, \mathbf{t}) := \{ v \in V(\mathbf{t}) \,:\, (\mathbf{t}, v) \in \text{sh}_0(\mathbf{t}, u)) \}
\label{eq:equivalent-nodes}
\end{align}
as the set of all nodes $v \in \mathcal{U}_n$ where rooting $\tilde{\mathbf{t}}_n$ at either node $u$ or $v$ yield the same rooted shape. In other words, $\text{Eq}(u, \tilde{\mathbf{t}}_n)$ is the set of all the nodes of $\tilde{\mathbf{t}}_n$ that are indistinguishable from $u$ once we remove the node labels. We give examples of indistinguishable nodes in Figure~\ref{fig:4nodes-example}.

Let $\mathbf{T}_n$ be a random recursive tree and suppose we observe that $\text{sh}(\mathbf{T}_n) = \mathbf{s}_n$ where $\mathbf{s}_n$ is a shape with an arbitrary alphabetically labeled representation $\tilde{\mathbf{t}}_n$. For any node $u \in \mathcal{U}_n$, the likelihood that any node in $\text{Eq}(u, \tilde{\mathbf{t}}_n)$ is the root is then the sum of the probabilities of all outcomes of the random recursive tree $\mathbf{T}_n$ that has the same rooted shape as $(\tilde{\mathbf{t}}_n, u)$. Since $\text{Eq}(u, \tilde{\mathbf{t}}_n)$ may contain multiple nodes, we then divide by the size of the set $\text{Eq}(u, \tilde{\mathbf{t}}_n)$ to obtain the likelihood of node $u$ being the root. 

To be precise, for a labeled tree $\tilde{\mathbf{t}}_n$ and a node $u \in \mathcal{U}_n$, we define
\[
\text{recur}(\tilde{\mathbf{t}}_n, u) := \{ \mathbf{t} \in \mathcal{T}^R_n \,:\, (\mathbf{t}, 1) \in \text{sh}_0(\tilde{\mathbf{t}}_n, u)\}
\]
which is the set of all distinct recursive trees that have the same rooted shape as $\tilde{\mathbf{t}}_n$ rooted at $u$. The likelihood of a node $u$ is then
\begin{align}
    \mathcal{L}(u, \tilde{\mathbf{t}}_n) &= \frac{1}{\# \text{Eq}(u, \tilde{\mathbf{t}}_n)} \sum_{\mathbf{t} \in \text{recur}(\tilde{\mathbf{t}}_n, u)} \mathbb{P}( \mathbf{T}_n = \mathbf{t})
    \label{eq:likelihood-definition}
\end{align}
where we divide by $\# \text{Eq}(u, \tilde{\mathbf{t}}_n)$ to account for multiplicity of indistinguishable nodes. If $\mathbf{T}_n$ is shape exchangeable, then $\mathbb{P}(\mathbf{T}_n = \mathbf{t})$ is a constant for any recursive tree $\mathbf{t}$ with the same shape as that of $\tilde{\mathbf{t}}_n$ and thus
\begin{align*}
\mathcal{L}(u, \tilde{\mathbf{t}}_n) \propto  \frac{ \# \text{recur}(\tilde{\mathbf{t}}_n, u)}{\# \text{Eq}(u, \tilde{\mathbf{t}}_n)} \qquad \text{(for shape exchangeable models)}
\end{align*}

The number of distinct recursive trees $\# \text{recur}(\tilde{\mathbf{t}}_n, u)$ is in general smaller than the number of distinct histories $\#\text{hist}(\tilde{\mathbf{t}}_n, u)$, see for example Figure~\ref{fig:4nodes-example}.  \citet[][Proposition~1]{bubeck2017finding} derives an exact expression of the $\# \text{recur}(\tilde{\mathbf{t}}_n, u)$ in terms of isomorphism classes of the subtrees of $\tilde{\mathbf{t}}_n$. This is difficult to compute however as it requires counting the automorphism classes of the subtrees.


\begin{figure}
    \centering
    \includegraphics[scale=.4]{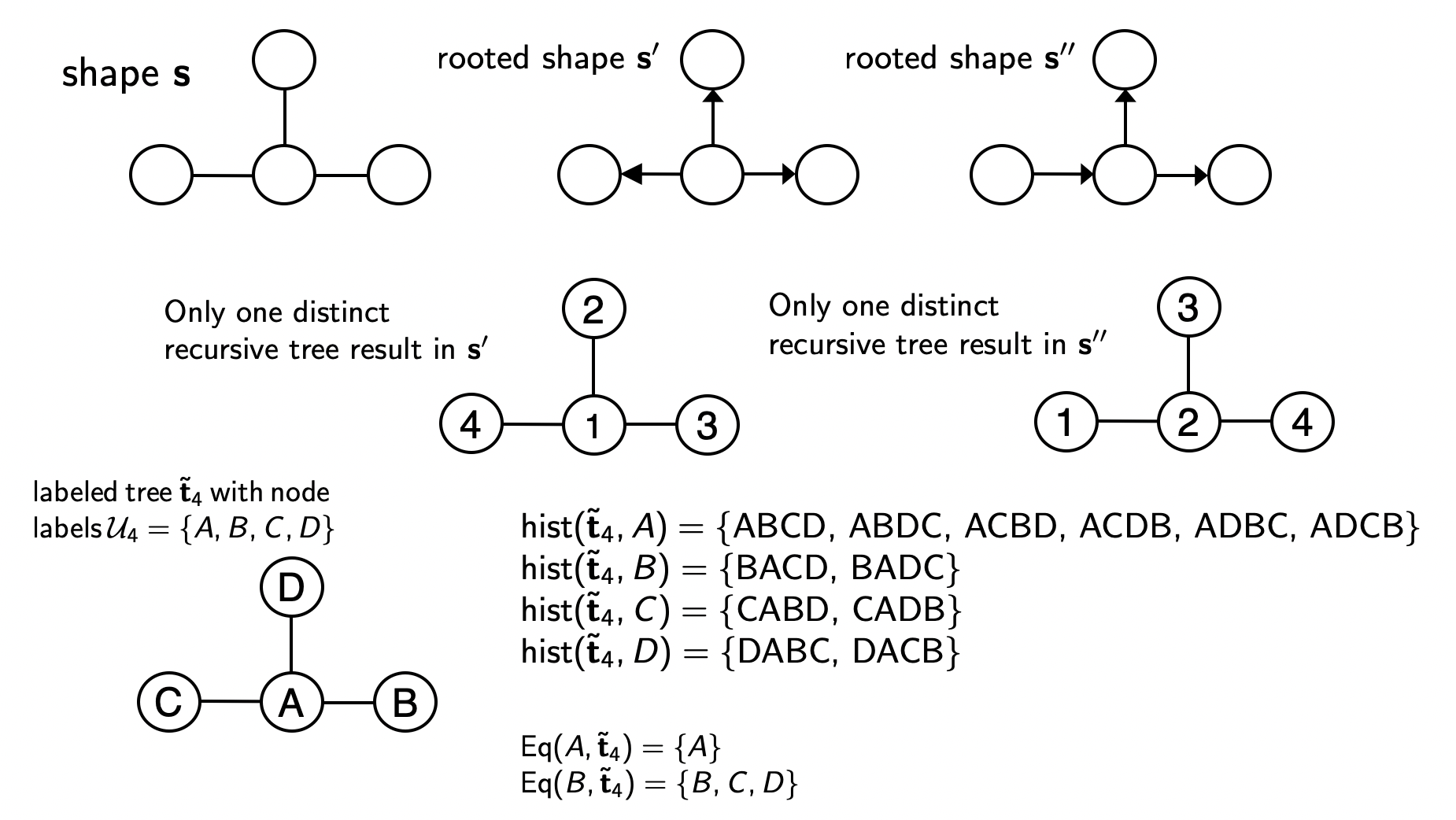}
    \caption{Example on a 4 node tree.}
    \label{fig:4nodes-example}
\end{figure}

Since the likelihood $\mathcal{L}(u, \tilde{\mathbf{t}}_n)$ is based on counting the recursive trees $\# \text{recur}(\tilde{\mathbf{t}}_n, u)$ and the conditional root probability $\mathbb{P}(\tilde{\mathbf{T}}_1 = \{u\} \,|\, \tilde{\mathbf{T}}_n = \tilde{\mathbf{t}}_n )$ is based on the counting the histories $\# \text{hist}(\tilde{\mathbf{t}}_n, u)$, it is not obvious how to relate the two. Our main result this section is the next proposition which shows that they in fact induce the same ordering of the nodes of a tree $\tilde{\mathbf{t}}_n$. We note this proposition strengthens the existing results in literature; in particular, it implies that \citet[][Theorem~5]{bubeck2017finding} applies in fact to the exact MLE instead of an approximate MLE as previously believed. 

\begin{theorem}
\label{thm:likelihood-equivalence}
Let $\mathbf{T}_n$ be a random recursive tree, not necessarily shape exchangeable, and let $\tilde{\mathbf{T}}_1, \ldots, \tilde{\mathbf{T}}_n$ be the corresponding random history. For any labeled tree $\tilde{\mathbf{t}}_n$ with node labels taking value in $\mathcal{U}_n$, we have that, for all $u \in \mathcal{U}_n$,
\[
\mathbb{P}(\tilde{\mathbf{T}} = \{u\} \,|\, \tilde{\mathbf{T}}_n = \tilde{\mathbf{t}}_n) = 
\frac{\mathcal{L}(u, \tilde{\mathbf{t}}_n)}{\sum_{v \in \mathcal{U}_n} \mathcal{L}(v, \tilde{\mathbf{t}}_n)} = 
\frac{1}{\# \text{Eq}(u, \tilde{\mathbf{t}}_n)} \mathbb{P}( (\mathbf{T}_n, 1) \in \text{sh}_0(\tilde{\mathbf{t}}_n, u) \,|\, \mathbf{T}_n \in \text{sh}(\tilde{\mathbf{t}}_n) ).
\]
\end{theorem}

We defer the proof of Theorem~\ref{thm:likelihood-equivalence} to Section~\ref{sec:likelihood-proof} in the appendix.

\begin{remark}
\label{rem:unlabeled-conditional}
From Theorem~\ref{thm:likelihood-equivalence}, we see that, for a labeled tree $\tilde{\mathbf{t}}_n$ and a node $u \in V(\tilde{\mathbf{t}}_n)$, we can define the conditional probability that $u$ is the root node with the quantity $\frac{1}{\# \text{Eq}(u, \tilde{\mathbf{t}}_n)} \mathbb{P}( (\mathbf{T}_n, 1) \in \text{sh}_0(\tilde{\mathbf{t}}_n, u) \,|\, \mathbf{T}_n \in \text{sh}(\tilde{\mathbf{t}}_n) )$, without the use of the label randomized tree $\tilde{\mathbf{T}}_n$. We divide by $\# \text{Eq}(u, \tilde{\mathbf{t}}_n)$ to address the issue that there may be multiple nodes that are indistinguishable from node $u$. Theorem~\ref{thm:likelihood-equivalence} shows that this is equivalent to the the label randomized conditional probability $\mathbb{P}( \tilde{\mathbf{T}}_1 = \{u\} \,|\, \tilde{\mathbf{T}}_n = \tilde{\mathbf{t}}_n)$. We use the latter expression for its simplicity. 
\end{remark}

\section{General Inference Problems}\label{section:general}

The approach that we take for root inference and the structural properties that we proved for shape exchangeable processes may also be used for general questions of interest about unobserved histories. To formalize general inference problems, we let $\mathbf{T}^*_n$ be the labeled representation of the observed shape $\text{sh}(\mathbf{T}_n)$ and let $\rho \in \text{Bi}([n], \mathcal{U}_n)$ be the label bijection that gives $\rho \mathbf{T}_n = \mathbf{T}_n^*$. Let $S$ be a discrete set and let $f_\rho : \mathcal{T}^R_n \rightarrow S$ be a function. For a confidence level $\epsilon \in (0, 1)$, our inference problem is to construct a confidence set $C_{\epsilon}(\mathbf{T}^*_n) \subset S$ such that
\begin{align}
\mathbb{P}\bigl( f_\rho(\mathbf{T}_{n}) \in C_{\epsilon}( \mathbf{T}^*_n) \bigr) \geq 1 - \epsilon. 
\label{eq:general-coverage}
\end{align}

\begin{example}
By taking $f_\rho(\mathbf{T}_n) = \text{root}_\rho(\mathbf{T}_n)$ and $S = \mathcal{U}_n$, we obtain the root inference problem (Section~\ref{sec:inference-problem}). A natural extension of the root inference problem is to infer the earliest $k$ nodes to appear, which is known as the seed-tree inference and has been studied by \cite{devroye2018discovery}; in seed-tree inference, for a given $k \in \mathbb{N}$, we let $S$ be the set of all subtrees of $\mathbf{T}^*_n$ of size $k$ and let $f_\rho(\mathbf{T}_n) = \text{seed}^{(k)}_\rho(\mathbf{T}_n) := \rho \mathbf{T}_k$. 


\end{example}

\begin{example}
For a fixed node $u \in \mathcal{U}_n$, we may infer the time at which it appeared in the observed tree. We let $S = [n]$ and define the random arrival time of node $u$ as 
\[
\text{Arr}^{(u)}_\rho( \mathbf{T}_{n}) := \min \{ t \in [n] \,:\, u \in V(\rho \textbf{T}_t) \} = \rho^{-1}(u).
\]
The problem is then to construct, for a given $\epsilon \in (0, 1)$, a confidence set $C^{(u)}_\epsilon(\mathbf{T}^*_n) \subset [n]$ such that the true arrival time $\text{Arr}^{(u)}_\rho(\mathbf{T}_n)$ is contained in $C^{(u)}_\epsilon$ with probability at least $1 - \epsilon$. This problem is well-defined for confidence sets that satisfy a labeling-equivariance condition; we defer the technical details to Section~\ref{sec:general-inference-supplement} of the appendix. 
\end{example}

We may consider a number of other inference problems, e.g., given a subset of nodes $\mathcal{V} \subset \mathcal{U}_n$, what is the order in which they were infected? The approach given below can be applied quite generally to most conceivable questions of this kind.  These may all be formalized in a manner identical to the examples that we have shown.  

In these cases, we replace an exact calculation of the conditional probability by Monte Carlo approximation. In particular, we derive two different computationally efficient sampling protocols to generate a history of a given tree uniformly at random. To make the discussion concrete, consider again the arrival time of a given node $u \in \mathcal{U}_n$: $\text{Arr}^{(u)}_{\rho}(\mathbf{T}_n) = \rho^{-1}(u) \in [n]$ for $\rho \in \text{Bi}([n], \mathcal{U}_n)$ such that $\rho \mathbf{T}_n = \mathbf{T}^*_n$. In this case, for a given tree $\tilde{\mathbf{t}}_n$, we may construct the confidence set by first computing  $B^{(u)}_{\epsilon}(\tilde{\mathbf{t}}_n)$ as the smallest subset of $[n]$ such that
\begin{align*}
\mathbb{P}( \Pi^{-1}(u) \in B^{(u)}_\epsilon(\tilde{\mathbf{t}}_n) \,|\, \tilde{\mathbf{T}}_n = \tilde{\mathbf{t}}_n ) 
&= \sum_{t \in B^{(u)}_\epsilon(\tilde{\mathbf{t}}_n)} \mathbb{P}( u \in \tilde{\mathbf{T}}_t \text{ and } u \notin \tilde{\mathbf{T}}_{t-1} \,|\, \tilde{\mathbf{T}}_n = \tilde{\mathbf{t}}_n ) \geq 1 - \epsilon.
\end{align*}
where $\Pi$ is a random bijection distributed uniformly in $\text{Bi}([n], \mathcal{U}_n)$ such that $\Pi \mathbf{T}_n = \tilde{\mathbf{T}}_n$ and where we take $\tilde{\mathbf{T}}_0 $ as the empty set. Following Theorem~\ref{thm:frequentist-coverage}, we may again show that $B^{(u)}_\epsilon(\tilde{\mathbf{t}}_n)$ has valid Frequentist coverage; we defer the formal statement and proof of this claim to Section~\ref{sec:general-inference-supplement}.

To compute the confidence set $B^{(u)}_\epsilon(\tilde{\mathbf{t}}_n)$, we need to compute, for each $t \in [n]$, the conditional probability
\begin{align}
\mathbb{P}( u \in \tilde{\mathbf{T}}_t \text{ and } u \notin \tilde{\mathbf{T}}_{t-1} \,|\, \tilde{\mathbf{T}}_n = \tilde{\mathbf{t}}_n ).
\label{eqn:conditional-arrival-time}
\end{align}
We propose a Monte Carlo approximation where we generate independent samples $\{ \tilde{\mathbf{T}}_1^{(m)}, \ldots, \tilde{\mathbf{T}}_{n-1}^{(m)} \}_{m=1}^M$ from the conditional history distribution $\mathbb{P}( \tilde{\mathbf{T}}_1, \ldots, \tilde{\mathbf{T}}_{n-1} \,|\, \tilde{\mathbf{T}}_n = \tilde{\mathbf{t}}_n)$. For an event $E$, we may then approximate
\[
\mathbb{P}( E \,|\, \tilde{\mathbf{T}}_n = \tilde{\mathbf{t}}_n) \approx \frac{1}{M} \sum_{m=1}^M \mathbbm{1}\bigl\{ \tilde{\mathbf{T}}_1^{(m)}, \ldots, \tilde{\mathbf{T}}_{n-1}^{(m)} \in E \bigr\}.
\]
For example, we may approximate the probability of node $u$ arriving at time $t$ (see~\eqref{eqn:conditional-arrival-time})
by 
\[
\frac{1}{M} \sum_{m=1}^{M} \mathbbm{1}\{ u \in \tilde{\mathbf{T}}_t^{(m)} \text{ and } u \notin \tilde{\mathbf{T}}_{t-1}^{(m)} \}.
\]
In the next section, we assume shape exchangeability and show two exact sampling schemes which allow us to efficiently carry out this approach. In Section~\ref{sec:importance_sampling}, we devise an importance sampling scheme for computing the conditional probability under a general process, not necessarily shape exchangeable. 


\subsection{Sampling for shape exchangeable processes}

If $\mathbf{T}_n$ is shape exchangeable, then we may generate a sample $\{ \tilde{\mathbf{T}}_1^{(m)},\ldots, \tilde{\mathbf{T}}_{n-1}^{(m)}\}$ by drawing a single history from the set $\text{hist}(\tilde{\mathbf{t}}_n)$ uniformly at random.  In general, the total number of histories is large, but uniform sampling can be carried out sequentially by either 
\begin{itemize}
    \item[(i)] {\em forward sampling}, which builds up a realization from the conditional distribution by a sequential process of adding nodes to the following scheme, or
    \item[(ii)] {\em backward sampling}, which recreates a history from the observed shape by sequentially removing nodes according to the correct conditional distributions.
\end{itemize}

Throughout this section, it will be convenient to think of a history $\tilde{\mathbf{T}}_1 \subset \tilde{\mathbf{T}}_2 \subset \ldots \subset \tilde{\mathbf{T}}_{n-1}$ as an ordered sequence $\mathbf{u} = (\mathbf{u}_1, \mathbf{u}_2, \ldots, \mathbf{u}_n) \in \mathcal{U}_n^n$ where $\tilde{\mathbf{T}}_k = \tilde{\mathbf{T}}_n \cap \{\mathbf{u}_1, \ldots, \mathbf{u}_k\}$ for every $k \in [n]$.

\subsubsection{Forward sampling}

Conditional on the final shape $\text{sh}(\mathbf{T}_n)=\mathbf{s}_n$, a uniform random history can be generated by an analog to the P\'olya urn process.  We generate the first element $\mathbf{u}_1$ (root) of the history from the distribution $\mathbb{P}( \tilde{\mathbf{T}}_1 = \cdot \,|\, \tilde{\mathbf{T}}_n = \tilde{\mathbf{t}}_n )$ over $\mathcal{U}_n$, where the conditional root probability distribution can be computed in $O(n)$ time through Algorithm~\ref{alg:count_history}. Once the root of the history is fixed, we then sequentially choose the next node $u$ by size-biased sampling based on the size of the subtree rooted root at $u$ away from the root $\mathbf{u}_1$.

More explicitly, let $\tilde{\mathbf{t}}_n$ be a tree labeled in $\mathcal{U}_n$ and rooted at $u_1\in\mathcal{U}_n$.  For any $v\in\mathcal{U}_n$, we write $\tilde{\mathbf{t}}_v^{(u_1)}$ to denote the subtree of $\tilde{\mathbf{t}}_n$ rooted at node $v$ away from $u_1$ and $n_v^{(u_1)}:=\#\mathbf{\tilde{t}}_v^{(u_1)}$ as the size of the subtree.  Given the shape of $\tilde{\mathbf{t}}_n$, we generate a history $(u_1,\ldots,u_n)\in\text{hist}(\mathbf{t}_n)$ by
\begin{itemize}
    \item sampling $u_1$ from $\mathbb{P}(\tilde{\mathbf{T}}_1=\cdot\mid\tilde{\mathbf{T}}_n=\tilde{\mathbf{t}}_n)$ and
    \item given $u_1,u_2,\ldots,u_{k-1}$, sampling $u_{k}$ from among the remaining nodes in $\mathcal{U}_n-\{u_1,\ldots,u_{k-1}\}$ according to
\begin{align}
\mathbb{P}( \mathbf{u}_k = u_k \,|\, \tilde{\mathbf{T}}_n =\tilde{\mathbf{t}}_n,\, \mathbf{u}_1 = u_1, \ldots, \mathbf{u}_{k-1} = u_{k-1}) =\left\{\begin{array}{cc}  \frac{n^{(u_1)}_{u_k}}{n-k+1}, & \{u_1,\ldots,u_k\}\in\text{hist}(\tilde{\mathbf{t}}_n,u_1),\\
0, & \text{otherwise.}
\end{array}\right.
\label{eq:k-node-history}
\end{align}
\end{itemize}
We note that~\eqref{eq:k-node-history} is a well-defined probability distribution because once we have fixed the first $k-1$ nodes of the history $(u_1, u_2, \ldots, u_{k-1})$, the probability on the left hand side of~\eqref{eq:k-node-history} is positive only for a neighbor $v$ of $(u_1, u_2, \ldots, u_{k-1})$. Summing $n^{(u_1)}_{v}$ over all neighbors $v$ of $(u_1, u_2, \ldots, u_{k-1})$ gives exactly the denominator $n-k+1$ of~\eqref{eq:k-node-history}.

The coming proposition shows that the result of this process is a valid draw from the conditional distribution of the history given $\sh(\tilde{\mathbf{T}}_n)$.

\begin{proposition}
\label{prop:remaining-history-prob}
Let $\mathbf{T}_n$ be a shape exchangeable preferential attachment process and $\tilde{\mathbf{T}}_n$ be a randomly labeled element of $\text{sh}(\mathbf{T}_n)$.    
For any sequence of nodes $(u_1, u_2, \ldots, u_n) \in \text{hist}(\tilde{\mathbf{t}}_n, u_1)$, the conditional distribution of $\mathbf{u}_k$ given $\{\tilde{\mathbf{T}}_n = \tilde{\mathbf{t}}_n,\, \mathbf{u}_1 = u_1, \ldots, \mathbf{u}_{k-1} = u_{k-1}\}$ the conditional distribution in \eqref{eq:k-node-history}.
\end{proposition}

\begin{proof} 
Let the labeled tree $\tilde{\mathbf{t}}_n$ be fixed and suppose $(\mathbf{u}_1, \mathbf{u}_2, \ldots, \mathbf{u}_{n-1})$ is a random history where $\mathbb{P}(\mathbf{u}_1 = u \,|\, \tilde{\mathbf{T}}_n = \tilde{\mathbf{t}}_n) = \mathbb{P}(\tilde{\mathbf{T}}_1 = \{u\} \,|\, \tilde{\mathbf{T}}_n = \tilde{\mathbf{t}}_n)$ for all $u \in \mathcal{U}_n$ and where the probabilities of $\mathbf{u}_2, \ldots, \mathbf{u}_{n-1}$ are specified by~\eqref{eq:k-node-history}. 

By Proposition~\ref{prop:count-rooted-history} and~\eqref{eq:conditional-root-probability}, we have that, for any fixed history $(u_1, u_2, \ldots, u_{n-1}) \in \text{hist}(\tilde{\mathbf{t}}_n)$,
\begin{align*}
&\mathbb{P}(\mathbf{u}_1 = u_1,\, \mathbf{u}_2 = u_2, \ldots, \mathbf{u}_{n-1} = u_{n-1} \,|\, \tilde{\mathbf{T}}_n = \tilde{\mathbf{t}}_n) \\
&= \mathbb{P}( \tilde{\mathbf{T}}_1 = \{u_1\} \,|\, \tilde{\mathbf{T}}_n = \tilde{\mathbf{t}}_n) \prod_{k=2}^{n-1} \mathbb{P}(\mathbf{u}_k = u_k \,|\, \tilde{\mathbf{T}}_n = \tilde{\mathbf{t}}_n,\, \mathbf{u}_1 = u_1, \ldots, \mathbf{u}_{k-1} = u_{k-1}) \\
&= \mathbb{P}( \tilde{\mathbf{T}}_1 = \{u_1\} \,|\, \tilde{\mathbf{T}}_n = \tilde{\mathbf{t}}_n) \prod_{k=2}^{n-1} \frac{n_{u_k}^{(u_1)}}{n-k+1} \\
&= \mathbb{P}( \tilde{\mathbf{T}}_1 = \{u_1\} \,|\, \tilde{\mathbf{T}}_n = \tilde{\mathbf{t}}_n) \frac{1}{(n-1)!} \prod_{v \in \mathcal{U}_n - \{u_1\}} n_{v}^{(u_1)} = \frac{1}{\#\text{hist}(\tilde{\mathbf{t}}_n)}
\end{align*}
as desired.
\end{proof}

Proposition~\ref{prop:remaining-history-prob} states that to generate the second node of the history given that the first node is $u_1$, we consider all neighbors $v_1, \ldots, v_{L(u_1)}$ of $u_1$ (where $L(u_1)$ denotes the number of neighbors of $u_1$) and choose $v_\ell$ with probability proportional to the size of the subtree $n^{(u_1)}_{v_\ell}$. Continuing in this way, once we have generated the first $k-1$ nodes of the history, we consider all neighbors $v_1, \ldots, v_{L(u_{1:(k-1)})}$ of the subtree $\tilde{\mathbf{t}}_n \cap \{u_1, \ldots, u_{k-1}\}$ and again choose a neighbor $v_\ell$ with probability proportional to the size of the subtree $n^{(u_1)}_{v_\ell}$. 

 The existence of such a sampling scheme for shape exchangeable processes is closely related to other sequential constructions for generating samples from exchangeable processes, such as the Chinese restaurant process found throughout the Bayesian nonparametrics and combinatorial probability literature; see, e.g., \cite{Crane2016ESF} for an overview of various constructions for exchangeable partition processes.  The size-biased sampling without replacement in the above can also be related to general Polya urn schemes which are known to produce exchangeable sequences. 

The computational complexity required to compute the size of all the subtrees $\{n_v^{(u_1)}\}_{v \in \mathcal{U}_n}$ is linear in total number of nodes $n$ because we can use a bottom-up procedure that makes a single pass through all of the nodes of the tree. Therefore, the computational complexity of generating the first $k$ elements of the history depends on the number of neighbors $L(u_1), L(u_{1:2}), \ldots, L(u_{1:(k-1)})$. The number of neighbors could be of order $O(n)$ and thus requiring $O(n^2)$ time to generate a full history $(u_1, \ldots, u_{(n-1)})$. 

We can improve the runtime of generating a full history to $O(n\log n)$ in the worst case in Algorithm~\ref{alg:draw-history}. The algorithm proceeds by drawing the first node of the history $u_1$ from distribution~\eqref{eq:conditional-root-probability}; we now view the input labeled tree $\tilde{\mathbf{t}}_n$ as being rooted at $u_1$. We then generate a random permutation of the node labels $\mathcal{U}_n \setminus \{u_1\}$ and modify the random permutation by swapping the position of a node $v$ with that of its parent $\text{pa}(v)$ if $v$ appears in the permutation before $\text{pa}(v)$. We continue until the random permutation satisfies the constraints necessary to be a valid history. The modification process can be done efficiently through sorting so that the overall runtime is at most $O(n \log \text{diam}(\tilde{\mathbf{t}}_n))$ where $\text{diam}(\tilde{\mathbf{t}}_n)$ is the length of the longest path (diameter) of the tree $\tilde{\mathbf{t}}_n$, as shown in the following proposition. We note that for uniform attachment or linear preferential attachment trees, the diameter is $O_p(\log n)$ (see e.g. \citet[][Theorem~6.32]{drmota2009random} and \citet[][Theorem~18]{bhamidi2007universal}) and so the runtime of Algorithm~\ref{alg:draw-history} is typically $O(n \log\log n)$.

\begin{algorithm}[htp]
\caption{Generating a random history uniformly from $\text{hist}(\tilde{\mathbf{t}}_n)$.}
\label{alg:draw-history}
\textbf{Input:} Labeled tree $\tilde{\mathbf{t}}_n$ whose node labels take value in $\mathcal{U}_n$.\\
\textbf{Output:} A history represented as a sequence $\mathbf{u}_1, \mathbf{u}_2, \ldots, \mathbf{u}_n \in \mathcal{U}_n$ where $\mathbf{u}_1 = u$.
\begin{algorithmic}[1]
\State Set $\mathbf{u}_1 = u$ with probability $\mathbb{P}(\tilde{\mathbf{T}}_1 = \{u\} \,|\, \tilde{\mathbf{T}}_n = \tilde{\mathbf{t}}_n) = \frac{\# \text{hist}(\tilde{\mathbf{t}}_n, u)}{\# \text{hist}(\tilde{\mathbf{t}}_n)}$ (c.f. Algorithm~\ref{alg:count_history}).
\State Generate a permutation $\mathbf{\Sigma} \,:\, \mathcal{U}_n \setminus \{\mathbf{u}_1\} \rightarrow \{2,3,\ldots,n\}$ uniformly at random.
\State Initialize a set $\mathcal{M} = \{ \mathbf{u}_1 \}$.
\For{$t = 2,3, \ldots, n$}:
  \State Let $v = \mathbf{\Sigma}^{-1}(t)$. If $v \in \mathcal{M}$,  continue to the next iteration.
  \State Otherwise, let $v_1 = v$, $v_2 = \text{pa}(v_1),\, \ldots, v_k = \text{pa}(v_{k-1})$ where $k$ is the largest integer such that $v_1, v_2, \ldots, v_k \notin \mathcal{M}$. \Comment{$\text{pa}(v)$ denotes the parent-node of $v$ with respect to $\tilde{\mathbf{t}}_n$ rooted at $u$.}
  \State Write $(t_1, \ldots, t_k) = \{ \mathbf{\Sigma}(v_1), \ldots, \mathbf{\Sigma}(v_k)$\} and order them such that $t_{(1)} \leq t_{(2)} \leq \ldots t_{(k)}$.
  \State Set $\mathbf{u}_{t_{(1)}} = v_k, \mathbf{u}_{t_{(2)}} = v_{k-1}, \ldots, \mathbf{u}_{t_{(k)}} = v_1$.
  \State Add $v_1, \ldots, v_k$ to $\mathcal{M}$.
\EndFor
\end{algorithmic}
\end{algorithm}

\begin{proposition}
\label{prop:runtime}
For any labeled tree $\tilde{\mathbf{t}}_n$ with node labels $\mathcal{U}_n$, Algorithm~\ref{alg:draw-history} generates a history $\mathbf{u}_1, \mathbf{u}_2, \ldots, \mathbf{u}_n$ uniformly at random from $\text{hist}(\tilde{\mathbf{t}}_n)$. Moreover, Algorithm~\ref{alg:draw-history} has a worst-case runtime of $O(n \log \text{diam}(\tilde{\mathbf{t}}_n))$ where $\text{diam}(\tilde{\mathbf{t}}_n)$ is the diameter, i.e. the length of the longest path, of $\tilde{\mathbf{t}}_n$.
\end{proposition}

\begin{proof}
It is clear that the output $\mathbf{u}_1, \ldots, \mathbf{u}_n$ is a valid history. Suppose $\mathbf{u}_1 = u_1$ for some $u_1 \in \mathcal{U}_n$, then $\mathbf{u}_2 = v$ for some neighbor $v$ of $u_1$ if and only if $\mathbf{\Sigma}(v') = 2$ for some node $v'$ in the subtree $\tilde{\mathbf{t}}_v^{(u_1)}$ rooted at $v$. Since $\mathbf{\Sigma}$ is a random permutation, the event that $\mathbf{\Sigma}(v') = 2$ for some $v' \in V(\tilde{\mathbf{t}}_v^{(u_1)})$ is exactly $\frac{n^{(u_1)}_v}{n-1}$.

Now assume that we have generated the first $k$ elements of the history $\mathbf{u}_1 = u_1,\, \mathbf{u}_2 = u_2,\ldots, \mathbf{u}_k = u_k$. Again, the event that $\mathbf{u}_{k+1} = v$ for some neighboring node $v$ of $\{u_1, \ldots, u_k\}$ if and only if $\mathbf{\Sigma}(v') = k+1$ for some $v' \in V(\tilde{\mathbf{t}}_v^{(u_1)})$, which occurs with probability exactly $\frac{n^{(u_1)}_v}{n-k+1}$. By Proposition~\ref{prop:remaining-history-prob}, it thus holds that $(\mathbf{u}_1, \mathbf{u}_2, \ldots, \mathbf{u}_{n-1})$ is uniform sample from $\text{hist}(\tilde{\mathbf{t}}_n)$

To prove the runtime of Algorithm~\ref{alg:draw-history}, we note that $\mathbf{\Sigma}$ may be generated in $O(n)$ time, for example by the Knuth-Fisher-Yates shuffle algorithm~\citep{fisher1943statistical}. For each node $v$, we let $k_v$ denote the number of ancestor nodes, including $v$ itself, not in the set $\mathcal{M}$ when the algorithm is at iteration $t = \mathbf{\Sigma}(v)$. It then holds that $\sum_{v \in \mathcal{U}_n \setminus \{u\}} k_v = n$ since each node is placed in the set $\mathcal{M}$ as soon as it is visited by the algorithm. Thus,
\[
\sum_{v \in \mathcal{U}_n \setminus \{u\}} k_v \log(k_v) \leq  \sum_{v \in \mathcal{U}_n \setminus \{u\}} k_v \log(\text{diam}(\tilde{\mathbf{t}}_n)) \leq n \log(\text{diam}(\tilde{\mathbf{t}}_n)),
\]
which proves the proposition. 
\end{proof}

\subsubsection{Backward sampling}

We may also generate a history $\tilde{\mathbf{T}}_1^{(m)},\ldots, \tilde{\mathbf{T}}_{n-1}^{(m)}$, equivalently represented as a sequence of nodes $u_1, u_2, \ldots, u_n$, from the conditional distribution $\mathbb{P}( \tilde{\mathbf{T}}_1, \ldots, \tilde{\mathbf{T}}_{n-1} \,|\, \tilde{\mathbf{T}}_n = \tilde{\mathbf{t}}_n)$ by iterative removing one leaf a time.

If $\mathbf{T}_n$ is shape exchangeable, then we have by Proposition~\ref{prop:uniform-conditional} that for any labeled tree $\tilde{\mathbf{t}}_n$ and any leaf node $u$ that
\begin{align}
\mathbb{P}( \tilde{\mathbf{T}}_{n-1} = \tilde{\mathbf{t}}_n \setminus \{u\} \,|\, \tilde{\mathbf{T}}_{n} = \tilde{\mathbf{t}}_n ) 
 = \frac{\# \text{hist}( \tilde{\mathbf{t}}_n - \{u\} )}{ \# \text{hist}(\tilde{\mathbf{t}}_n )}. \label{eqn:remove-leaf-prob}
\end{align}

Thus, given an observed shape $\mathbf{s}_n$ and with an alphabetically labeled representation $\tilde{\mathbf{t}}_n$, we may reconstruct a history leading to that shape by recursively removing leaves according to~\eqref{eqn:remove-leaf-prob}. Naive implementation of backward sampling has a runtime of $O(n^3)$ but, by observing that 
\begin{align*}
&\mathbb{P}( \tilde{\mathbf{T}}_{n-1} = \tilde{\mathbf{t}}_n \backslash \{u \} | \tilde{\mathbf{T}}_{n} = \tilde{\mathbf{t}}_n ) \\
&\qquad \qquad = \sum_{v \in \mathcal{U}_n} \mathbb{P}( \tilde{\mathbf{T}}_{n-1} = \tilde{\mathbf{t}}_n \backslash \{u \} | \tilde{\mathbf{T}}_{n} = \tilde{\mathbf{t}}_n, \tilde{\mathbf{T}}_1 = \{v\} )  \mathbb{P}( \tilde{\mathbf{T}}_1 = \{v\}| \tilde{\mathbf{T}}_{n} = \tilde{\mathbf{t}}_n) \\
&\qquad \qquad = \sum_{v \in \mathcal{U}_n} \frac{ \# \text{hist}( \tilde{\mathbf{t}}_n - \{ u \}, v)}{\# \text{hist}(\tilde{\mathbf{t}}_n, v)} \frac{ \# \text{hist}(\tilde{\mathbf{t}}_n, v)}{ \text{hist}(\tilde{\mathbf{t}}_n) },
\end{align*}
one may first draw a random node $v$ from the distribution $\mathbb{P}( \tilde{\mathbf{T}}_1 = \cdot \,|\, \tilde{\mathbf{T}}_{n} = \tilde{\mathbf{t}}_n)$ and then remove a random leaf node $u$ drawn with probability $\frac{ \# \text{hist}( \tilde{\mathbf{t}}_n - \{ u \}, v)}{\# \text{hist}(\tilde{\mathbf{t}}_n, v)}$. The quantities $\# \text{hist}( \tilde{\mathbf{t}}_n - \{ u \}, v)$ for all leaf node $u$ can be computed jointly in $O(n \cdot \text{diam}(\tilde{\mathbf{t}}_n))$ time by precomputing the size of the subtrees. The overall runtime of backward sampling is therefore $O(n^2 \text{diam}(\tilde{\mathbf{t}}_n))$, which is typically $O(n^2 \log n)$ for uniform attachment or linear preferential attachment random trees (c.f. comment right before Proposition~\ref{prop:runtime}).

In general, it may be possible to stop backward sampling before generating the entire tree history. For example, for determining the event of which of two nodes was infected first, the event is determined as soon as one of these nodes is removed during the reversed process. 

\subsection{Importance sampling for general tree models}
\label{sec:importance_sampling}

If $\mathbf{T}_n$ is not shape exchangeable, then the conditional distribution~\eqref{eq:conditional-history-prob} is possibly intractable to directly compute or sample. In the case where the process is not shape exchangeable but where the model is known and the underlying probabilities are straightforward to compute, we propose an importance sampling approach where we use the shape exchangeable conditional distribution as the proposal distribution. More precisely, for a given observed shape $\textbf{s}_n$ with a labeled representation $\tilde{\textbf{t}}_n$, we generate samples of the history $\tilde{\mathbf{T}}_1^{(m)}, \ldots, \tilde{\mathbf{T}}_{n-1}^{(m)}$ for $m=1,\ldots, M$ uniformly from $\text{hist}(\tilde{\mathbf{t}}_n)$ and associate each sample with an importance weight.

For a history $\tilde{\mathbf{t}}_1, \tilde{\mathbf{t}}_2, \dots, \tilde{\mathbf{t}}_{n-1}$ of a labeled tree $\tilde{\mathbf{t}}_n$, we write $\mathbb{P}_{\text{unif}}(\tilde{\mathbf{T}}_1 = \tilde{\mathbf{t}}_1, \ldots, \tilde{\mathbf{T}}_{n-1} = \tilde{\mathbf{t}}_{n-1} \,|\, \tilde{\mathbf{T}}_n = \tilde{\mathbf{t}}_n) = \frac{1}{\#\text{hist}(\tilde{\mathbf{t}}_n)}$ as uniform distribution over $\text{hist}(\tilde{\mathbf{t}}_n)$ and $\mathbb{P}(\tilde{\mathbf{T}}_1 = \tilde{\mathbf{t}}_1, \ldots, \tilde{\mathbf{T}}_{n-1} = \tilde{\mathbf{t}}_{n-1} \,|\, \tilde{\mathbf{T}}_n = \tilde{\mathbf{t}}_n)$ as the actual probability. The importance weight for a single history $\tilde{\mathbf{t}}_1, \ldots, \tilde{\mathbf{t}}_{n-1}$ is defined as the ratio of the actual probability over the proposal probability. In our setting, we observe that
\begin{align}
w(\tilde{\mathbf{t}}_1, \ldots, \tilde{\mathbf{t}}_{n-1}) &= \frac{\mathbb{P}(\tilde{\mathbf{T}}_1 = \tilde{\mathbf{t}}_1, \ldots, \tilde{\mathbf{T}}_{n-1} = \tilde{\mathbf{t}}_{n-1} \,|\, \tilde{\mathbf{T}}_n = \tilde{\mathbf{t}}_n)}
{\mathbb{P}_{\text{unif}}(\tilde{\mathbf{T}}_1 = \tilde{\mathbf{t}}_1, \ldots, \tilde{\mathbf{T}}_{n-1} = \tilde{\mathbf{t}}_{n-1} \,|\, \tilde{\mathbf{T}}_n = \tilde{\mathbf{t}}_n)} \nonumber
\\
&= 
\mathbb{P}(\tilde{\mathbf{T}}_1 = \tilde{\mathbf{t}}_1, \ldots, \tilde{\mathbf{T}}_{n-1} = \tilde{\mathbf{t}}_{n-1}, \tilde{\mathbf{T}}_n = \tilde{\mathbf{t}}_n) \frac{\#\text{hist}(\tilde{\mathbf{t}}_n)}{\mathbb{P}(\tilde{\mathbf{T}}_n = \tilde{\mathbf{t}}_n)} \label{eqn:importance-weight} \\
&\propto \mathbb{P}(\tilde{\mathbf{T}}_1 = \tilde{\mathbf{t}}_1, \ldots, \tilde{\mathbf{T}}_{n-1} = \tilde{\mathbf{t}}_{n-1}, \tilde{\mathbf{T}}_n = \tilde{\mathbf{t}}_n), \nonumber
\end{align}
where the last line follows because the $\#\text{hist}(\tilde{\mathbf{t}}_n) / \mathbb{P}( \tilde{\mathbf{T}}_n = \tilde{\mathbf{t}}_n )$ term of~\eqref{eqn:importance-weight} depends only on the final tree $\tilde{\mathbf{t}}_n$ and not on the history $\tilde{\mathbf{t}}_1, \ldots, \tilde{\mathbf{t}}_{n-1}$. Since the importance weights only need to be specified up to a multiplicative constant, we may compute the actual probability $\mathbb{P}(\tilde{\mathbf{T}}_1 = \tilde{\mathbf{t}}_1, \ldots, \tilde{\mathbf{T}}_{n-1} = \tilde{\mathbf{t}}_{n-1}, \tilde{\mathbf{T}}_n = \tilde{\mathbf{t}}_n)$ as the weights. 

To give a concrete example of importance sampling, we consider the problem of computing the conditional root probability $\mathbb{P}(\tilde{\mathbf{T}}_1 = \{v\} \,|\, \tilde{\mathbf{T}} = \tilde{\mathbf{t}}_n)$ for a node $v \in \mathcal{U}_n$ when $\mathbf{T}_n$ is not shape exchangeable. Our first step is to generate $M$ samples of histories $\{\tilde{\mathbf{t}}_1^{(m)}, \ldots, \tilde{\mathbf{t}}_{n-1}^{(m)} \}$ uniformly from $\text{hist}(\tilde{\mathbf{t}}_n)$ using Algorithm~\ref{alg:draw-history} and compute, for each sample, the importance weight 
\begin{align}
w_m := \mathbb{P}(\tilde{\mathbf{T}}_1 = \tilde{\mathbf{t}}^{(m)}_1, \ldots, \tilde{\mathbf{T}}_{n-1} = \tilde{\mathbf{t}}^{(m)}_{n-1}, \tilde{\mathbf{T}}_n = \tilde{\mathbf{t}}_n). \label{eq:importance-weight2}
\end{align}
We may then take the Monte Carlo approximation of the conditional root probability as
\[
\frac{1}{\sum_m w_m} \sum_{m=1}^M w_m \mathbbm{1}\{ \tilde{\mathbf{t}}_1^{(m)} = v \}.
\]

Since any conditional distribution $\mathbb{P}(\tilde{\mathbf{T}}_1, \ldots, \tilde{\mathbf{T}}_{n-1} \,|\, \tilde{\mathbf{T}}_n = \tilde{\mathbf{t}}_n)$ is supported on $\text{hist}(\tilde{\mathbf{t}}_n)$ and thus dominated by the uniform distribution, we immediately obtain by the law of large numbers that the importance sampling protocol converges to the true conditional probability as the number of samples $M$ goes to infinity.
\begin{proposition}
Let $\mathbf{T}_n$ be any random recursive tree and let $\tilde{\mathbf{T}}_n$ be the corresponding label-randomized tree. For any labeled tree $\tilde{\mathbf{t}}_n$ with $V(\tilde{\mathbf{t}}_n) = \mathcal{U}_n$, for any event $E \subset \text{hist}(\tilde{\mathbf{t}}_n)$, we have that
\[
\mathbb{P}\bigl( \tilde{\mathbf{T}}_1, \ldots, \tilde{\mathbf{T}}_{n-1} \in E \,|\, \tilde{\mathbf{T}} = \tilde{\mathbf{t}}_n \bigr) = \lim_{M \rightarrow \infty} \frac{1}{\sum_m w_m} \sum_{m=1}^M w_m \mathbbm{1}\{ \tilde{\mathbf{T}}_1^{(m)}, \ldots, \tilde{\mathbf{T}}^{(m)}_{n-1} \in E \},
\]
where $w_m$ is defined as in~\eqref{eq:importance-weight2}.
\end{proposition}

\begin{remark}
After completion of this manuscript, some recent independent work on sampling algorithms for random tree histories was brought to the authors' attention.   \citet{cantwell2019recovering} have independently observed the same sampling probability as in Proposition \ref{prop:remaining-history-prob} for the forward sampling algorithm and \citet{young2019phase} have independently proposed a similar importance sampling scheme which is also adaptive. 
\end{remark}

\section{Experiments}\label{sec:experiments}

\subsection{Simulation studies}
\label{sec:simulation}
As a simple illustration our root inference procedure, we show in Figure~\ref{fig:bigpa} a tree of 1000 nodes generated from the linear preferential attachment model. We construct the 95\% confidence set, comprising of around 30 nodes colored green and cyan as well as the 85\% confidence set, comprising of 6 nodes colored green. The true root node is colored yellow and is captured in the 95\% confidence set. In Figure~\ref{fig:bigua}, we show the same except that the tree is generated from the uniform attachment model.  

\begin{figure}
    \begin{center}
    \begin{subfigure}{0.45\textwidth}
    \includegraphics[scale=.64, trim=150 60 0 80 ]{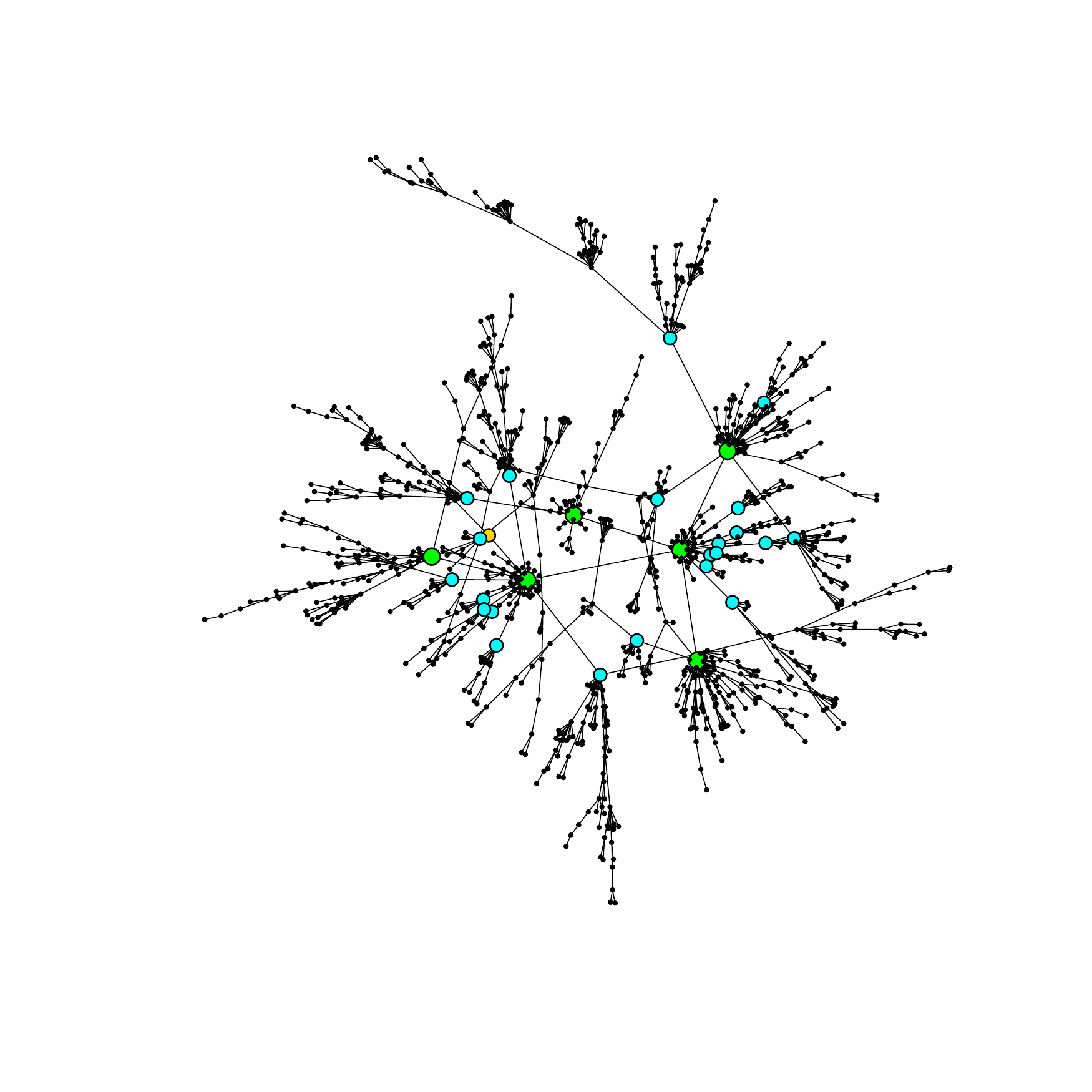}
    \caption{Root inference on linear preferential attachment tree; green nodes form 85\% confidence set; green and cyan nodes form 95\% confidence set. }
    \label{fig:bigpa}
    \end{subfigure}
    \begin{subfigure}{0.45\textwidth}
    \includegraphics[scale=.64, trim=80 60 0 80]{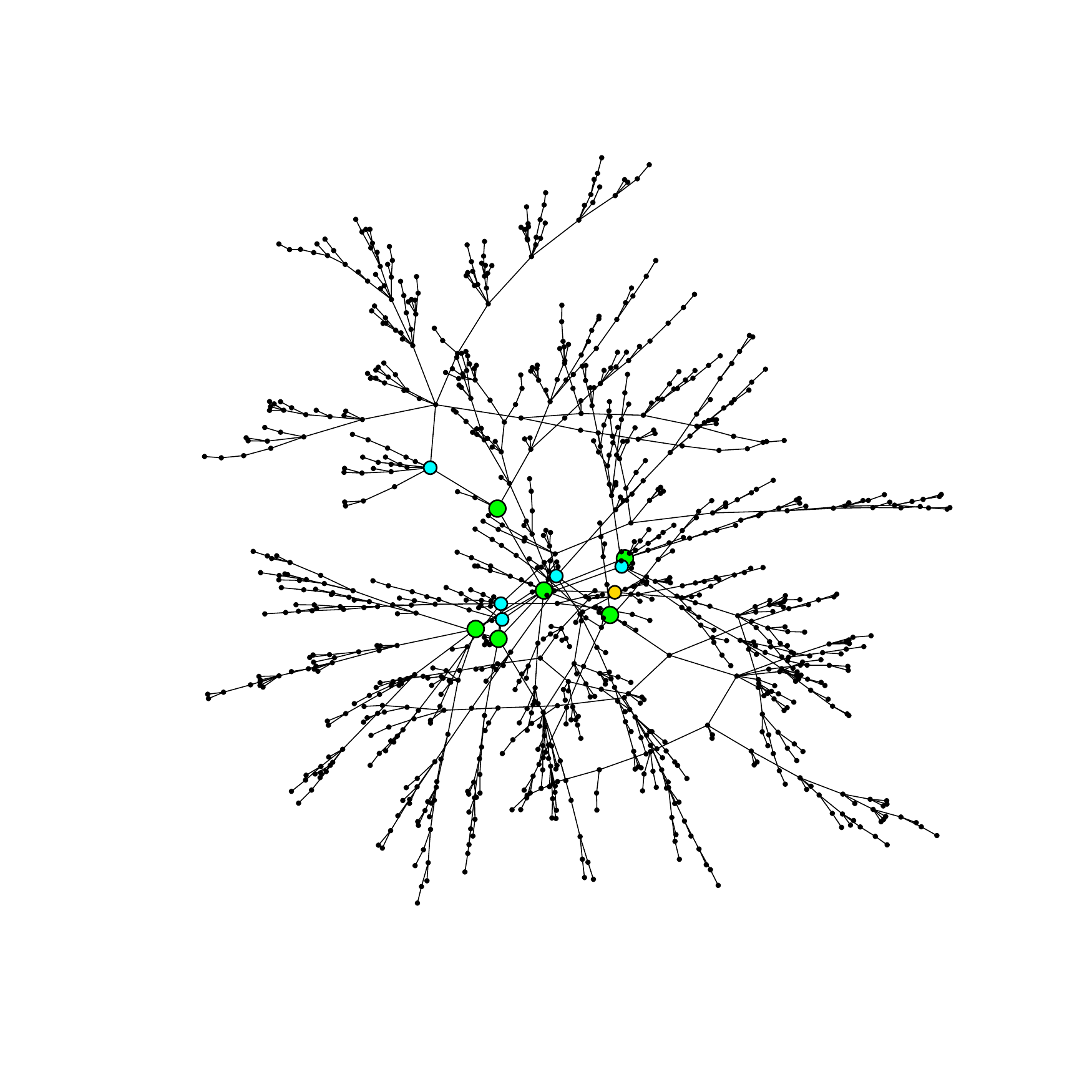}
    \caption{Root inference on uniform attachment tree; green nodes form 85\% confidence set; green and cyan nodes form 95\% confidence set.}
    \label{fig:bigua}
    \end{subfigure}
    \end{center}
    \caption{}
\end{figure}

In our first set of simulations studies, we verify that our confidence set for the root node has the frequentist coverage predicted by the theory. We generate trees from the general preferential attachment model $\text{PA}_\phi$ where we take $\phi(d) = d, 1, 8+d, 8-d$ which correspond to linear preferential attachment, uniform attachment, affine preferential attachment, and uniform on $8$-regular tree models respectively. We generate 200 independent trees and then calculate our confidence sets and report the percentage of the trials where our confidence set captures the true root node. We summarize the results in Table~\ref{tab:coverage}. Our findings are in full agreement with our theory, showing that we indeed attain valid coverage.

\begin{table}
\begin{center}
\begin{tabular}{c|c|c|c|c|c|c|c|c|}
\hline
\text{$\phi(d) = $} & d & 1 & 8 + d & 8 - d& d & 1 & d & 1 \\
\hline
\text{Theoretical coverage} & 0.95 & 0.95 & 0.95 & 0.95 & 0.9 & 0.9 & 0.99 & 0.99 \\
\hline
{\bf \text{Empirical coverage}} & {\bf 0.955} & {\bf 0.955} & {\bf 0.95} & {\bf 0.935} & {\bf 0.895} & {\bf 0.895} & {\bf 1} & {\bf 1} \\
\hline
\end{tabular}
\end{center}
\caption{Empirical coverage of our confidence set for the root node. We report the average over 200 trials. Tree size is $10,000$ in all cases.}
\label{tab:coverage}
\end{table}

In our second set of simulation studies, we analyze the size of our confidence sets for the root node. First, we generate trees from the linear preferential attachment model where we vary the tree size from $n=5,000$ to $n=100,000$.  We then compute the $95\%$ confidence sets and report the average size of the confidence set as well as the standard deviation from 200 independent trials. We summarize the result in Table~\ref{tab:conf_size1}. We observe that, in accordance with Corollary~\ref{cor:confidence_set_size}, the size of our confidence sets does not increase with the size of the tree. 

Next, we perform the same experiment on linear preferential attachment trees except that we hold the tree size constant at $n=10,000$ and instead vary the size from the confidence level from $0.90$ to $0.95$ to $0.99$. We report the average size of the confidence set as well as the standard deviation from 200 independent trials. We summarize the results in Table~\ref{tab:conf_size2}. To compare with these results, we also compute the size of the confidence set given by the bound $C \log(1/\epsilon)^2/\epsilon^4$ from \citet[][Theorem~6]{bubeck2017finding}. The constant $C$ arises from complicated approximations. We use $C=0.23$ as a conservative lower bound and justify this bound in Section~\ref{sec:constant-justification} in the appendix. We find that the bound, though theoretically beautiful, yield confidence sets that are far too conservative to be useful. 

We then analyze the size of the confidence sets under the uniform attachment model. We use the same setting where we let the tree size be $n=10,000$ and summarize the results in Table~\ref{tab:conf_size3}. We observe that under linear preferential attachment model, the size of the confidence set increases much more with the confidence level than under the uniform attachment model. This is in accordance with Corollary~\ref{cor:confidence_set_size} and the theoretical analysis of  \cite{bubeck2017finding}. We also compare the size of our confidence sets with the bound of $2.5 \log(1/\epsilon)/\epsilon$ that arises from~\citet[][Theorem~4]{bubeck2017finding}. We note that \cite{bubeck2017finding} also gives a bound of $a \exp\bigl( \frac{b \log(1/\epsilon)}{\log \log 1/\epsilon} \bigr)$ but this is far too large for any conservative values of $a,b$. 

\begin{table}[htp]
    \centering
    \begin{tabular}{c|c|c|c|c|}
    \hline
    \text{Number of nodes} & 5,000 & 10,000 & 20,000 & 100,000  \\
    \hline
   \textbf{Size of confidence set} &  {\bf 31.31} $\pm$  11.55 &  {\bf 34.23} $\pm$ 13.6 &  {\bf 35.85} $\pm$ 15 &  {\bf 36.68} $\pm$ 12.5 \\
    \hline
    \end{tabular}
    \caption{Size of the $95\%$ confidence set under linear preferential attachment model. We report the mean and standard deviation over 200 trials.}
    \label{tab:conf_size1}
\end{table}

\begin{table}[htp]
    \centering
    \begin{tabular}{c|c|c|c|}
    \hline
    \text{Confidence level} & 0.90 & 0.95 & 0.99  \\
    \hline
    \text{Bound in BDL (2017)} 
    & 12,194 & 330,258 & $\geq$ 48 million \\
    \hline
   {\bf \text{Size of confidence set}} &  {\bf 13.98} $\pm$ 5.1 &  {\bf 34.23} $\pm$ 13.6 &  {\bf 193.8} $\pm$ 60.8 \\
    \hline
    \end{tabular}
    \caption{Size of the confidence set under linear preferential attachment model for a tree of $n=10,000$ nodes. We also give the best bound on size from \citet[][Theorem~6]{bubeck2017finding} of $C \log^2(1/\epsilon)/\epsilon^4$ (letting $C=0.23$ as a conservative bound, see Section~\ref{sec:constant-justification}) for comparison purpose. }
    \label{tab:conf_size2}
\end{table}

\begin{table}[htp]
    \centering
    \begin{tabular}{c|c|c|c|}
    \hline
    \text{Confidence level} & 0.90 & 0.95 & 0.99  \\
    \hline
    \text{Bound in BDL (2017)} & 57 & 150 & 1,151 \\
    \hline
   {\bf \text{Size of confidence set}} &  {\bf 7.6} $\pm$ 0.68 &  {\bf 12.43} $\pm$ 1.2 &  {\bf 29.65} $\pm$ 2.7 \\
    \hline
    \end{tabular}
    \caption{Size of the confidence set under uniform attachment model for a tree of $n=10,000$ nodes. We also give the best bound on size from \citet[][Theorem~3]{bubeck2017finding} of $2.5 \log(1/\epsilon)/\epsilon$ for comparison purpose.}
    \label{tab:conf_size3}
\end{table}

Next, we illustrate Algorithm~\ref{alg:draw-history} for sampling from the uniform distribution on the set of histories of a labeled tree. We generate a single tree of 300 nodes from the linear preferential attachment model, shown in Figure~\ref{fig:arrival-tree}. We select three nodes, colored red, blue, and green and we draw 500 samples from the conditional distribution of the history to infer the conditional distribution of arrival times of these three nodes. The true arrival time of the red node is 3, of the blue node is 50, and of the green node is 200; the true root node is shown in yellow. The inferred conditional distribution of arrival times is shown in Figure~\ref{fig:arrival-time}. We observe that the conditional distribution of the arrival times reflect the "centrality" of these three high-lighted nodes. 

\begin{figure}
    \begin{center}
    \begin{subfigure}{0.45\textwidth}
    \includegraphics[scale=.6, trim=140 80 0 80]{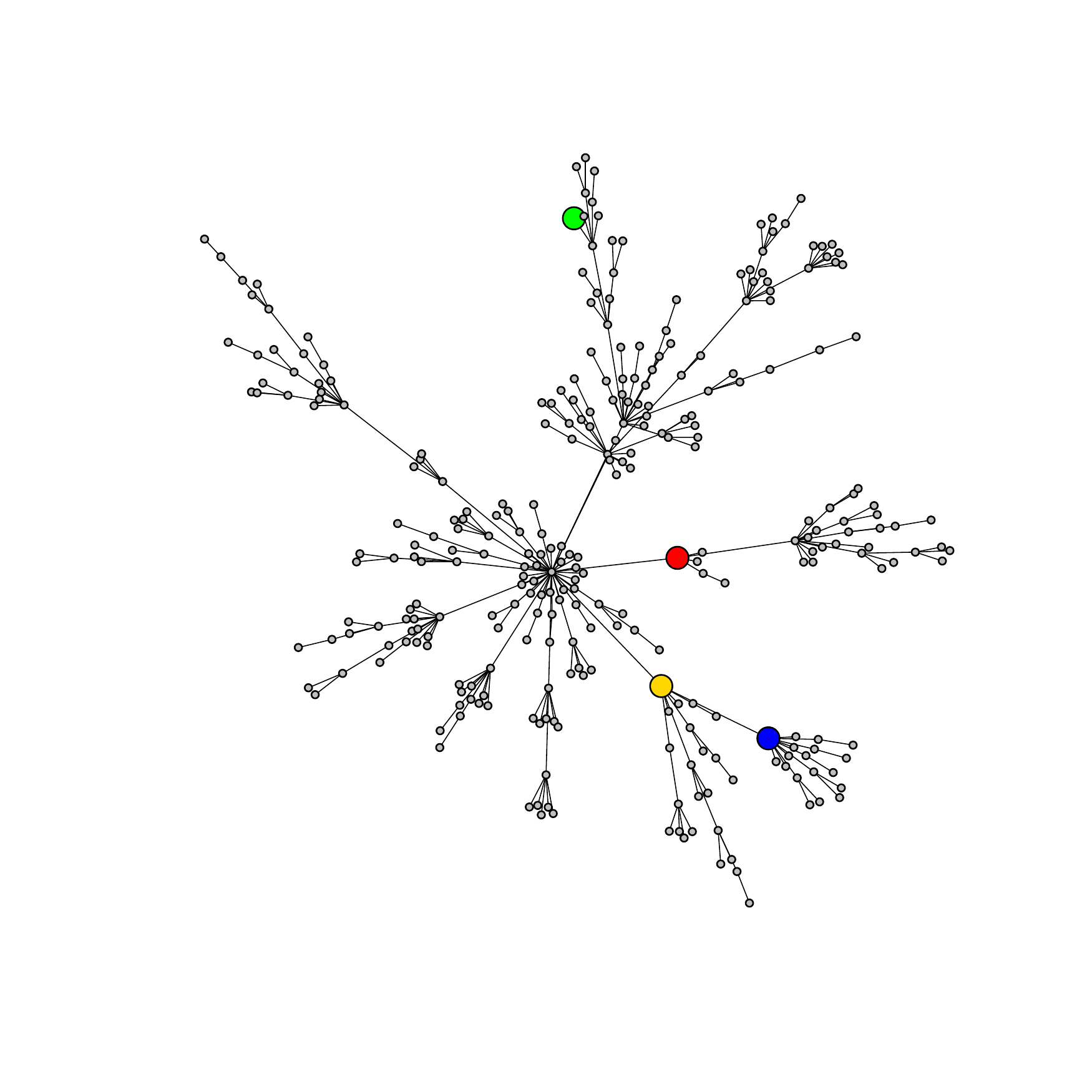}
    \caption{Linear PA Tree of 300 nodes.}
    \label{fig:arrival-tree}
    \end{subfigure}
    \begin{subfigure}{0.45\textwidth}
    \includegraphics[scale=.4, trim=0 0 100 0]{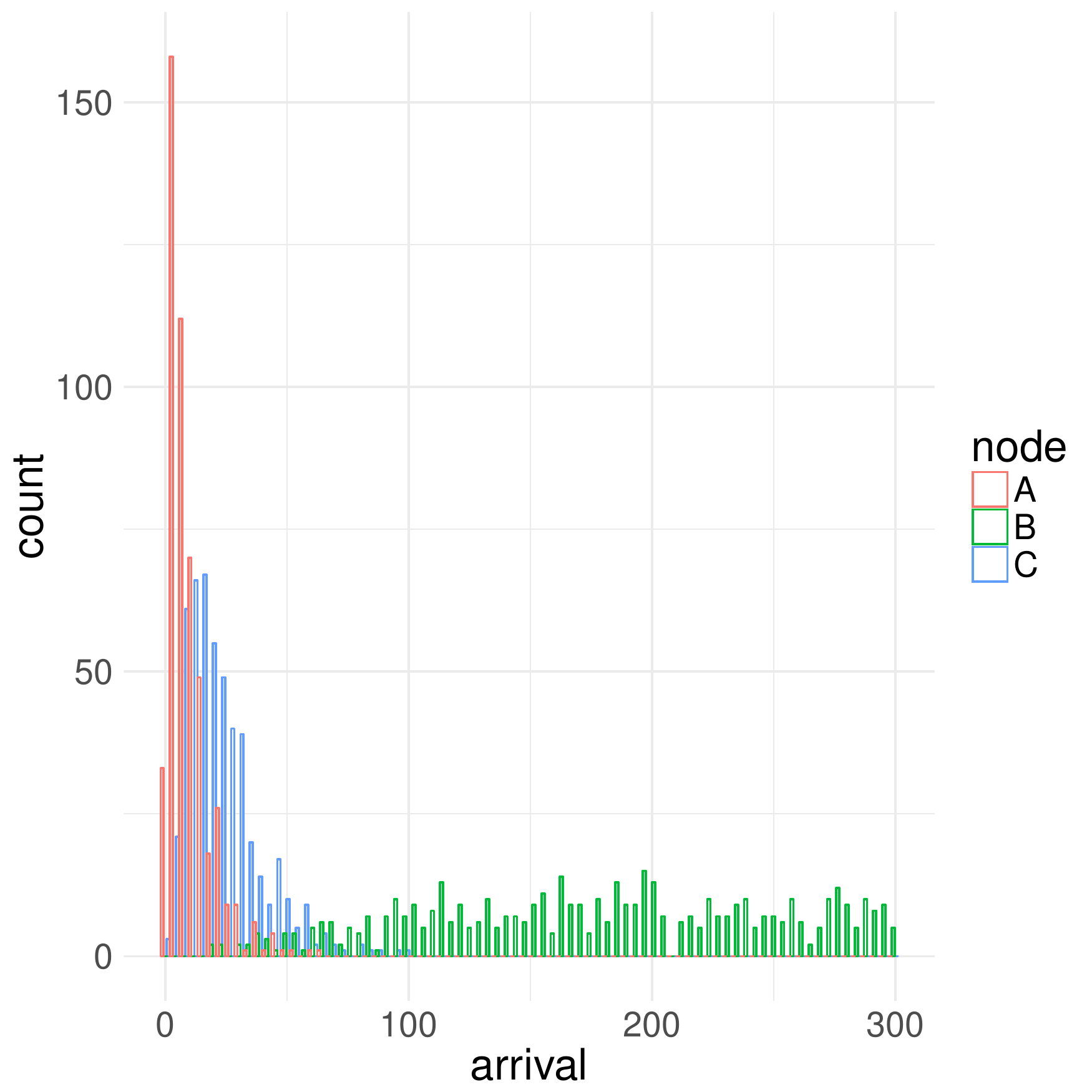}
    \caption{Conditional distribution of the arrival times of the three colored nodes.}
    \label{fig:arrival-time}
    \end{subfigure}
    \end{center}
    \caption{}
\end{figure}

\subsection{Flu data}
\label{sec:flu}
In this section, we run our method on a flu transmission network from \cite{hens2012robust}. The data set originates from an A(H1N1)v flu outbreak in a London school in April 2009. The patient-zero was a student who returned from travel abroad. After the outbreak, researchers used contact tracing to reconstruct a network of inter-personal contacts between 33 pupils in the same class as patient-zero, depicted in Figure~\ref{fig:flu-net} where patient-zero is colored yellow. Using knowledge of the true patient-zero, times of symptom onset among all the infected students, and epidemiological models, \cite{hens2012robust} reconstructed a plausible infection tree, which is shown in Figure~\ref{fig:flu-tree}. 

We first consider the plausible infection tree reconstructed by \cite{hens2012robust} and see if we can determine the patient-zero from only the connectivity structure of the tree alone. We assume that the observed tree is shape exchangeable and apply the root inference procedure described in Section~\ref{sec:root-inference}. We construct the 95\% confidence set, which comprises the group of 10 nodes colored green (and patient-zero colored yellow), as well as the 85\% confidence set, which comprises of 4 nodes with the conditional root probability labels in red. The true patient-zero, colored yellow, is the node with the third highest conditional root probability and it is captured by both confidence sets. 

Next, we study the contact network in Figure~\ref{fig:flu-net}. The network is highly non-tree-like and so we first reduce it to the tree case by generating a random spanning tree where we generate a random Gaussian weight on each edge and then take the minimum spanning tree via Kruskal's algorithm (we note that this is not the uniform random spanning tree). We then apply our root inference procedure on the random spanning tree and compute the 95\% and the 85\% confidence set. We repeat this procedure 200 times (with 200 independent random spanning trees) and report the average sizes of the confidence sets as well as the coverage in Table~\ref{tab:flu}. We observe that although the random spanning trees are not necessary shape exchangeable, our root inference procedure is still able to provide useful output. We believe that we can use the same approach to perform history inference on a randomly growing network, not necessarily a tree; we defer a detailed study of this approach to future work. 

\begin{figure}
    \begin{center}
    \begin{subfigure}{0.45\textwidth}
    \includegraphics[scale=.5, trim=120 60 100 100]{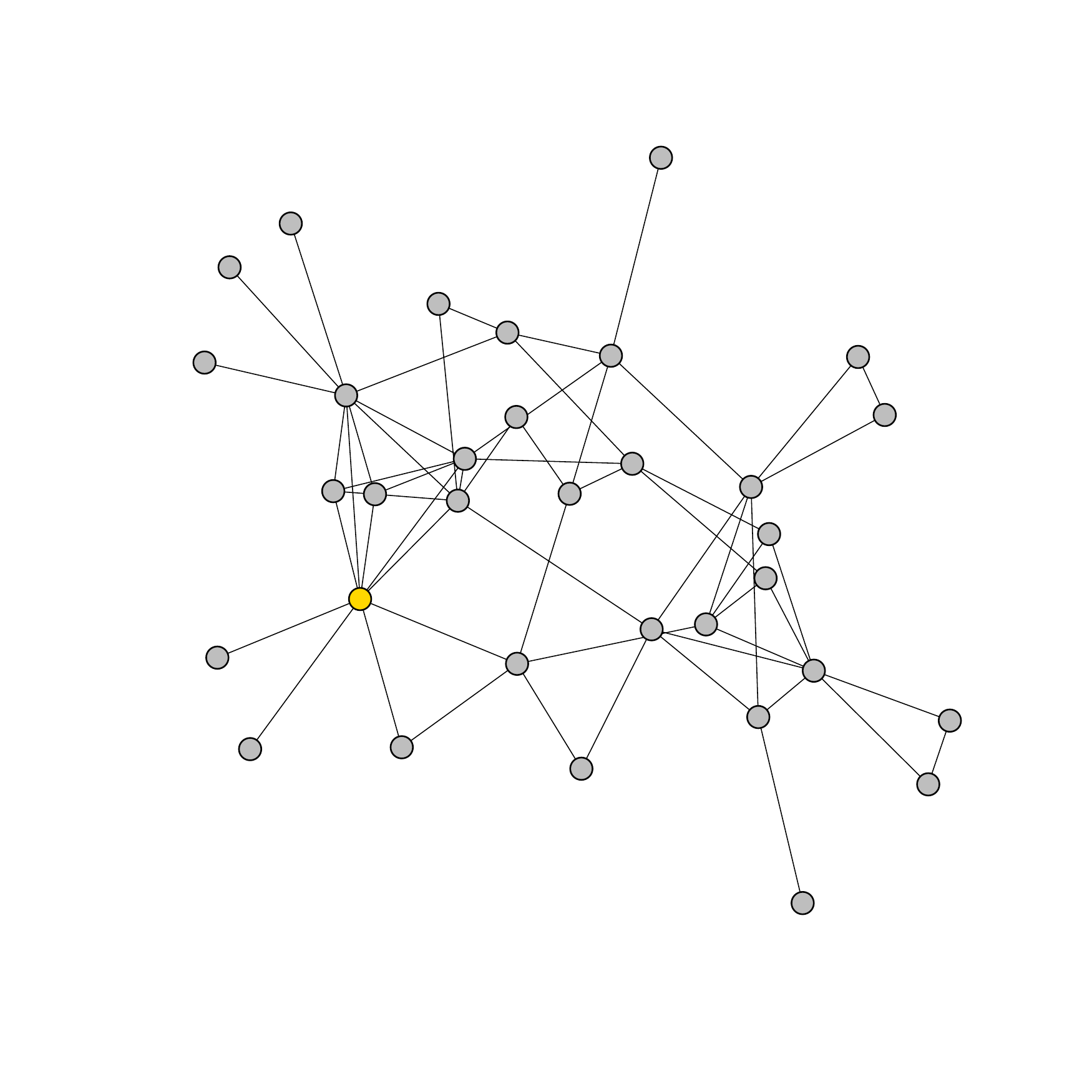}
    \caption{Contact network among 33 students with patient-zero colored yellow.}
    \label{fig:flu-net}
    \end{subfigure}
    \begin{subfigure}{0.45\textwidth}
    \includegraphics[scale=.5, trim=50 80 0 80]{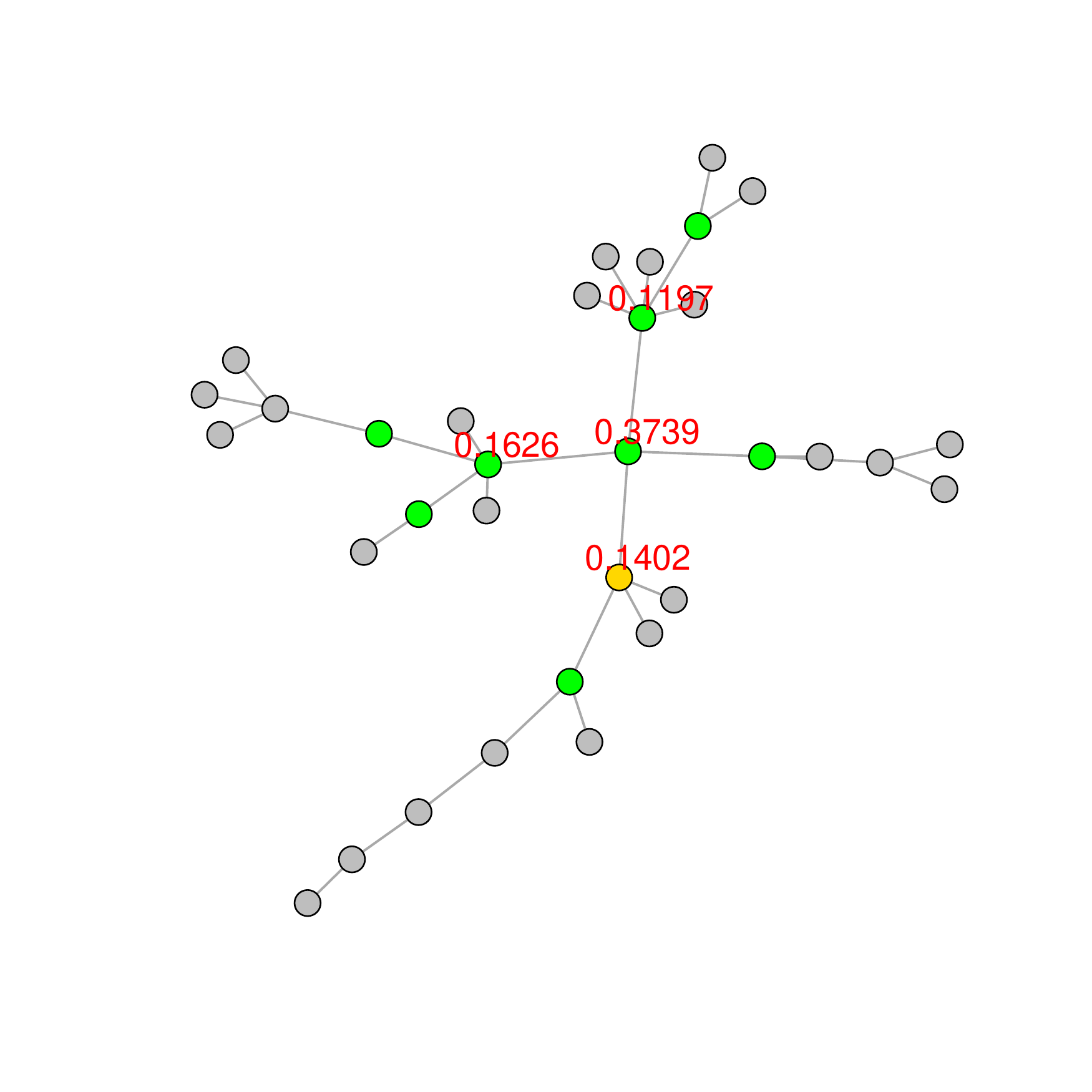}
    \caption{Plausible infection tree reconstructed by \cite{hens2012robust} with confidence sets.}
    \label{fig:flu-tree}
    \end{subfigure}
    \end{center}
    \caption{}
\end{figure}

\begin{table}[htp]
    \centering
    \begin{tabular}{|c|c|c|}
    \hline
    \text{Confidence level} & \textbf{Average size} & \textbf{Empirical Coverage} \\
    \hline
    0.85 & {\bf 4.84} & {\bf 0.895} \\
    \hline
    0.95 & {\bf 9.755} & {\bf 1} \\
    \hline
    \end{tabular}
    \caption{Result of root inference on 200 random spanning trees of the contact network in Figure~\ref{fig:flu-net}.}
    \label{tab:flu}
\end{table}

\section{Discussion}
\label{section:discussion}

In this paper, we consider the specific setting where the shape of the infection tree is known but the infection ordering is unobserved and must be inferred. In many real world applications such as contact tracing, which is used for infectious disease containment, we do not observe the exact infection tree but rather a network of interactions among a group of individuals. In these cases, our methods may be applied as a heuristic on a spanning tree of the observed network. We defer a careful study of inference on the history of a general network to future work. 

Another open question is how to incorporate side information that are often present with the edges. For example, in contact tracing, each edge may be associated with a time stamp of when that edge was formed. The time stamp may be noisy because of the patients being interviewed may not remember the timing of the interactions perfectly but the information could still be valuable in providing a more precise inferential result. 

\section{Acknowledgement}

The second author would like to thank Alexandre Bouchard-Cot\'{e} and Jason Klusowski for helpful conversations. The authors would also like to thank Jean-Gabriel Young for pointing us to some recent related work in the physics community and Tauhid Zaman for providing some additional references. The authors further acknowledge anonymous reviewers for valuable comments and suggestions. This work is partially supported by NSF Grant DMS-1454817. 

\bibliographystyle{dcu}
\bibliography{reference}

\clearpage

\setcounter{section}{0}
\setcounter{equation}{0}
\setcounter{theorem}{0}
\def\theequation{S\arabic{section}.\arabic{equation}}
\def\thesection{S\arabic{section}}
\def\thetheorem{S\arabic{theorem}}

\begin{center}
\Large{Supplementary material to `Inference for the History of a Randomly Growing Tree'} \\ \vspace{0.2in}
\large{Harry Crane and Min Xu}
\end{center}

\section{Proof of Theorem~\ref{thm:pa-shape-exchangeable}}
\label{sec:proof-pa-shape}

\begin{proof}
First suppose that $\phi(d)= \max(0, \alpha+\beta d)$ for $\alpha,\beta$ satisfying the conditions of the theorem.  For $n\geq1$ and an arbitrary recursive tree $\mathbf{t}_n$, let $\mathbf{t}_k:=\mathbf{t}_n\cap[k]$ and let $u_k\in[k-1]$ be the parent node of $k$ for every $k\geq2$. Then we have
\begin{align}
\mathbb{P}(\mathbf{T}_n = \mathbf{t}_n) &=
  \prod_{k=2}^n \frac{\phi(\text{deg}(u_k, \mathbf{t}_{k-1}))}{\sum_{v \in [k-1]} \phi(\text{deg}(v, \mathbf{t}_{k-1}))} \nonumber \\
&= \prod_{k=2}^n \frac{\alpha + \beta \text{deg}(u_k, \mathbf{t}_{k-1})}{\sum_{v \in [k-1]} \alpha + \beta \text{deg}(v, \mathbf{t}_{k-1})} \nonumber \\
&= \prod_{k=2}^{n} \frac{1}{(k-1)\alpha + 2 (k-2) \beta} \prod_{k=2}^n (\alpha + \beta \text{deg}(u_k, \mathbf{t}_{k-1}) )\\
&=\prod_{k=2}^n\frac{1}{(k-1)\alpha+2(k-2)\beta}\prod_{v\in V(\mathbf{t}_n)}(\alpha+\beta)(\alpha+2\beta)\cdots(\alpha+\beta(\text{deg}(v,\mathbf{t}_n)-1)), \label{eq:pa_prob}
\end{align}
where the penultimate equality follows because $\mathbf{t}_{k-1}$ has $k-2$ edges and where the final equality follows because for every node $v \in [n]$ such that $d := \text{deg}(v, \mathbf{t}_n) > 1$, $v$ is attached to a new node at times $k_1, k_2, \ldots, k_{d-1} \subset [n]$ and the degree of $v$ at time $k_i$ is $i$. Because the distribution of $\mathbf{T}_n$ depends on $\mathbf{T}_n$ only through its degree distribution, which is a measurable function of $\text{sh}(\mathbf{T}_n)$, it follows that $\mathbf{T}_n$ is shape exchangeable.




For the converse, suppose that $\mathbf{T}_n$ is from a PA$_{\phi}$ process and is shape exchangeable.  
For any pair of nodes $u$ and $v$ with degree $k$ and $k'$ and at least one leaf node each in a realization $\mathbf{t}_n$, consider removing a leaf from both nodes to obtain $\mathbf{t}^*_{n-2}$ with resulting degree distribution $d^*$.  And now consider adding the $(n-1)$st and $n$th node to $\mathbf{t}^*_{n-2}$ in order to obtain the final state $\mathbf{t}_n$. By shape exchangeability, both orderings must result in the same probability.

Let $\Phi(d^*):=\sum_{w\neq u,v}\phi(\text{deg}(w,\mathbf{t}^*_{n-2}))$ be the total weight to nodes of $\mathbf{t}^*_{n-2}$ other than $u$ and $v$. 
In the case where the $(n-1)$st node connects to $u$ and the $n$th connects to $v$, the conditional probability is
\[\frac{\phi(k-1)}{\Phi(d^*)+\phi(k-1)+\phi(k'-1)}\frac{\phi(k'-1)}{\Phi(d^*)+\phi(k)+\phi(k'-1)}.\]
And in the case where the $(n-1)$st node connects $v$ and the $n$th connects to $u$, the conditional probability is
\[\frac{\phi(k'-1)}{\Phi(d^*)+\phi(k-1)+\phi(k'-1)}\frac{\phi(k-1)}{\Phi(d^*)+\phi(k-1)+\phi(k')}.\]
Shape exchangeability forces
\begin{eqnarray*}
\lefteqn{\frac{\phi(k-1)}{\Phi(d^*)+\phi(k-1)+\phi(k'-1)}\frac{\phi(k'-1)}{\Phi(d^*)+\phi(k)+\phi(k'-1)}=}\\&&=\frac{\phi(k'-1)}{\Phi(d^*)+\phi(k-1)+\phi(k'-1)}\frac{\phi(k-1)}{\Phi(d^*)+\phi(k-1)+\phi(k')},\end{eqnarray*}
from which it immediately follows that $\phi$ satisfies
\[\phi(i+1)+\phi(j-1)=\phi(i)+\phi(j)\]
for all $i,j\geq1$, and thus
\[\phi(i+1)-\phi(i)=\phi(j)-\phi(j-1)=\beta\]
for all $i\geq1$ and $j\geq2$ such that $\phi(i)>0$ and $\phi(j-1)>0$.
It follows that $\phi(i)-\phi(1)=i\beta$ and therefore must have the form
\[\phi(i)=\alpha+i\beta,\quad i\geq1,\]
for $\phi(1)\equiv\alpha$.  Finally, we must have $\phi(d)\geq0$ for all $d$ to ensure that the function $\phi$ determines a valid probability distribution.

\end{proof}

\section{Proof of Lemma~\ref{lem:comparison}}
\label{sec:size-proof}

\begin{proof}

Let $\mathbf{T}_n$ be a random recursive tree, let $\mathbf{T}^*_n$ be an alphabetically labeled representation of the observed shape $\text{sh}(\mathbf{T}_n)$ and let $\rho \in \text{Bi}([n], \mathcal{U}_n)$ be an isomorphism such that $\rho \mathbf{T}_n = \mathbf{T}^*_n$. Fix $\epsilon, \delta \in (0,1)$ and suppose that $C_{\delta \epsilon}(\cdot)$ is a labeling-equivariant (see Remark~\ref{rem:label-equivariance}) confidence set for the root node with asymptotic coverage $1-\delta \epsilon$, that is, $\liminf_{n \rightarrow \infty} \mathbb{P}( \text{root}_\rho(\mathbf{T}_n) \in C_{\delta \epsilon}(\mathbf{T}^*_n)) \geq 1 - \delta \epsilon$.


Let $\Pi$ be a random permutation drawn uniformly from $\text{Bi}([n], \mathcal{U}_n)$ and write $\tilde{\mathbf{T}}_n := \Pi \mathbf{T}_n$ as the randomly labeled tree. Then, there exists a real-valued sequence $\mu_n \rightarrow 0$ such that
\begin{align}
&\mathbb{P}(\tilde{\mathbf{T}}_1 \in C_{\delta \epsilon}(\tilde{\mathbf{T}}_n) ) \nonumber \\
&= \sum_{\pi \in \text{Bi}([n], \mathcal{U}_n)} \mathbb{P}( \Pi(1) \in C_{\delta \epsilon}(\Pi \mathbf{T}_n) \,|\, \Pi = \pi) \mathbb{P}( \Pi = \pi) \nonumber \\
&= \mathbb{P}( \rho(1) \in C_{\delta \epsilon}(\rho \mathbf{T}_n) ) \nonumber \\
&= \mathbb{P}( \text{root}_\rho(\mathbf{T}_n) \in C_{\delta \epsilon}(\rho \mathbf{T}_n)) \geq 1 - \delta \epsilon + \mu_n, \label{eq:random-coverage}
\end{align}
where the penultimate equality follows from the labeling-equivariance of $C_{\delta \epsilon}(\cdot)$.

For any labeled tree $\tilde{\mathbf{t}}_n$, we have from definition~\eqref{eq:size-K-definition2} that $B_{\epsilon}(\tilde{\mathbf{t}}_n)$ is the smallest labeling-equivariant subset of $\mathcal{U}_n$ such that $\mathbb{P}( \tilde{\mathbf{T}}_1 \in B_{\epsilon}(\tilde{\mathbf{t}}_n) \,|\, \tilde{\mathbf{T}}_n = \tilde{\mathbf{t}}_n ) \geq 1- \epsilon$. Then, if $K_\epsilon(\tilde{\mathbf{t}}_n) > \# C_{\delta \epsilon}(\tilde{\mathbf{t}}_n)$, then it must be that $\mathbb{P}( \tilde{\mathbf{T}}_1 \in C_{\delta \epsilon}(\tilde{\mathbf{T}}_n) \,|\, \tilde{\mathbf{T}}_n = \tilde{\mathbf{t}}_n ) < 1- \epsilon$.

Therefore, we have from~\eqref{eq:random-coverage} that
\begin{align*}
1 - \delta \epsilon + \mu_n &\leq \mathbb{P}(\tilde{\mathbf{T}}_1 \in C_{\delta \epsilon}(\tilde{\mathbf{T}}_n) ) \\
&= \sum_{\tilde{\mathbf{t}}_n \in \tilde{\mathcal{T}}_n} 
\mathbb{P}( \tilde{\mathbf{T}}_1 \in C_{\delta \epsilon}(\tilde{\mathbf{T}}_n) \,|\, \tilde{\mathbf{T}}_n = \tilde{\mathbf{t}}_n ) \mathbb{P}(\tilde{\mathbf{T}}_n = \tilde{\mathbf{t}}_n ) \\
&\leq \mathbb{P}\bigl( K_\epsilon(\tilde{\mathbf{t}}_n) \leq \# C_{\delta \epsilon}(\tilde{\mathbf{t}}_n) \bigr) + (1 - \epsilon) \mathbb{P}\bigl( K_\epsilon(\tilde{\mathbf{t}}_n) > \# C_{\delta \epsilon}(\tilde{\mathbf{t}}_n) \bigr).
\end{align*}

We then obtain by algebra that
\[
\mathbb{P}(K_\epsilon(\tilde{\mathbf{t}}_n) > \# C_{\delta \epsilon}(\tilde{\mathbf{t}}_n)) \leq \delta + \mu_n / \epsilon,
\]
which yields the desired conclusion.
\end{proof}

\section{Proof of Theorem~\ref{thm:likelihood-equivalence}}
\label{sec:likelihood-proof}

\begin{proof}
Before proceeding to the proof, we first establish some helpful notation. 

For any labeled trees $\mathbf{t}, \mathbf{t}'$, not necessarily recursive, we define the set of isomorphisms as
\[
I(\mathbf{t}, \mathbf{t}') := \{ \pi \in \text{Bi}(V(\mathbf{t}), V(\mathbf{t}')) \,:\, \pi \mathbf{t} = \mathbf{t}'\}.
\]
And, for $u \in V(\mathbf{t})$ and $v \in V(\mathbf{t}')$, we also define the restricted set of isomorphisms as
\[
I(\mathbf{t},u, \mathbf{t}',v) := \{ \pi \in \text{Bi}(V(\mathbf{t}), V(\mathbf{t}')) \,:\, \pi \mathbf{t} = \mathbf{t}',\, \pi(u) = v\}.
\]
We note that $I(\mathbf{t}, \mathbf{t})$ is the set of automorphisms of $\mathbf{t}$. 

We have the following facts:
\begin{enumerate}
\item[Fact 1] $I(\mathbf{t}, \mathbf{t}')$ is non-empty if and only if $\mathbf{t}, \mathbf{t}'$ have the same shape. Moreover, the cardinality of $I(\mathbf{t}, \mathbf{t}')$ depends only on that shape. 
\item[Fact 2] $I(\mathbf{t}, u, \mathbf{t}', u')$ is non-empty if and only if $(\mathbf{t}, u)$ and $(\mathbf{t}', u')$ have the same rooted shape and the cardinality of $I(\mathbf{t}, u, \mathbf{t}', u')$ depends only on that rooted shape. As a consequence,  $I(\mathbf{t}, u, \mathbf{t}, v)$ is non-empty if and only if $v \in \text{Eq}(u, \mathbf{t})$.
\end{enumerate}

Fix $(\mathbf{t}, u)$ and $(\mathbf{t}', u')$ and let us suppose that they have the same rooted shape. If $\# \text{Eq}(u, \mathbf{t}) = 1$, then $\# I(\mathbf{t}, u, \mathbf{t}', u') = \# I(\mathbf{t}, \mathbf{t'})$. In general, we have that
\[
\# I(\mathbf{t}, \mathbf{t'}) = \sum_{u'' \in V(\mathbf{t'})} \# I(\mathbf{t}, u, \mathbf{t'}, u'') = \# I(\mathbf{t}, u, \mathbf{t}', u') \# \text{Eq}(u, \mathbf{t}).
\]

Recall that any history $\tilde{\mathbf{t}}_1 \subset \tilde{\mathbf{t}}_2 \subset \ldots \subset \tilde{\mathbf{t}}_n$ can be represented as a pair $(\mathbf{t}_n, \pi)$ where $\mathbf{t}_n$ is a recursive tree such that $\text{sh}(\mathbf{t}_n) = \text{sh}(\tilde{\mathbf{t}}_n)$ and where $\pi$ is a bijection from $[n]$ to $\mathcal{U}_n$. Similarly, any pair $(\mathbf{t}_n, \pi)$ can be represented as a history by taking $\tilde{\mathbf{t}}_k = \pi \mathbf{t}_k$ for all $k \in [n]$.

We then have that \begin{align}
\pi \in I(\mathbf{t}_n, 1, \tilde{\mathbf{t}}_n, u) \text{ if and only if } (\mathbf{t}_n, \pi) \in \text{hist}(\tilde{\mathbf{t}}_n, u).
\label{eqn:rooted_isomorphism}
\end{align}

Let $\Pi$ be a random bijection distributed uniformly in $\text{Bi}([n], \mathcal{U}_n)$, independently of $\mathbf{T}_n$,  such that $\tilde{\mathbf{T}}_n = \Pi \mathbf{T}_n$. We have, for any $\tilde{\mathbf{t}}_n$ with $V(\tilde{\mathbf{t}}_n) = \mathcal{U}_n$ and $u \in \mathcal{U}_n$, 
\begin{align*}
   & \sum_{(\mathbf{t}_n, \pi) \in \text{hist}(\tilde{\mathbf{t}}_n, u)} \mathbb{P}( \tilde{\mathbf{T}}_1 = \pi \mathbf{t}_1, \ldots, \tilde{\mathbf{T}}_{n-1} = \pi \mathbf{t}_{n-1}, \tilde{\mathbf{T}}_n = \pi \mathbf{t}_n) \\
   &=   \sum_{(\mathbf{t}_n, \pi) \in \text{hist}(\tilde{\mathbf{t}}_n, u)} \mathbb{P}( \mathbf{T}_1 = \mathbf{t}_1, \ldots, \mathbf{T}_{n-1} = \mathbf{t}_{n-1}, \mathbf{T}_n = \mathbf{t}_n) \mathbb{P}( \Pi = \pi) \\
    &= \frac{1}{n!} \sum_{(\mathbf{t}_n, \pi) \in \text{hist}(\tilde{\mathbf{t}}_n, u)} \mathbb{P}( \mathbf{T}_1 = \mathbf{t}_1, \ldots, \mathbf{T}_{n-1} = \mathbf{t}_{n-1}, \mathbf{T}_n = \mathbf{t}_n) \\
    &= \frac{1}{n!} \sum_{\mathbf{t}_n \in \text{recur}(\tilde{\mathbf{t}}_n, u)} \mathbb{P}(\mathbf{T}_n = \mathbf{t}_n) \#I(\mathbf{t}_n, 1, \tilde{\mathbf{t}}_n, u) \\
    &= \frac{\# I(\tilde{\mathbf{t}}_n, \tilde{\mathbf{t}}_n)}{\# \text{Eq}(u, \tilde{\mathbf{t}}_n) n!} \sum_{\mathbf{t}_n \in \text{recur}(\tilde{\mathbf{t}}_n, u)} \mathbb{P}(\mathbf{T}_n = \mathbf{t}_n) \\
    &= \frac{\# I(\tilde{\mathbf{t}}_n, \tilde{\mathbf{t}}_n)}{n!} \mathcal{L}(u, \tilde{\mathbf{t}}_n).
\end{align*}

Thus,
\begin{align*}
    \mathbb{P}(\tilde{\mathbf{T}}_1 = \{u\} \,|\, \tilde{\mathbf{T}}_n = \tilde{\mathbf{t}}_n) &=
    \frac{ \sum_{(\mathbf{t}_n, \pi) \in \text{hist}(\tilde{\mathbf{t}}_n,u)} \mathbb{P}( \tilde{\mathbf{T}}_1 = \pi \mathbf{t}_1 = u, \ldots, \tilde{\mathbf{T}}_{n-1} = \pi \mathbf{t}_{n-1}, \tilde{\mathbf{T}}_n = \pi \mathbf{t}_n) } %
    { 
     \sum_{(\mathbf{t}_n, \pi) \in \text{hist}(\tilde{\mathbf{t}}_n)} \mathbb{P}( \tilde{\mathbf{T}}_1 = \pi \mathbf{t}_1, \ldots, \tilde{\mathbf{T}}_{n-1} = \pi \mathbf{t}_{n-1}, \tilde{\mathbf{T}}_n = \pi \mathbf{t}_n) 
    }\\
    &= \frac{\mathcal{L}(u, \tilde{\mathbf{t}}_n)}{ \sum_{v \in \mathcal{U}_n} \mathcal{L}(v, \tilde{\mathbf{t}}_n)}.
\end{align*}

The final equality in the statement of the Theorem follows from the observation that
\[
\frac{\mathbb{P}( (\mathbf{T}_n, 1) \in \text{sh}_0(\tilde{\mathbf{t}}_n, u))}{\#\text{Eq}(u, \tilde{\mathbf{t}}_n)} = \frac{1}{\# \text{Eq}(\tilde{\mathbf{t}}_n, u)}  \sum_{\mathbf{t}_n \in \text{recur}(\tilde{\mathbf{t}}_n, u)} \mathbb{P}(\mathbf{T}_n = \mathbf{t}_n) = \mathcal{L}(u, \tilde{\mathbf{t}}_n).
\]
and that
\[
\mathbb{P}(\mathbf{T}_n \in \text{sh}(\tilde{\mathbf{t}}_n) ) = \sum_{v \in \mathcal{U}_n} \frac{\mathbb{P}( (\mathbf{T}_n, 1) \in \text{sh}_0(\tilde{\mathbf{t}}_n, v))}{\# \text{Eq}(v, \tilde{\mathbf{t}}_n)} = \sum_{v \in \mathcal{U}_n} \mathcal{L}(v, \tilde{\mathbf{t}}_n),
\]
where we divide by the size of the equivalent node class to adjust for double counting. The theorem then follows as desired. 
\end{proof}

\section{Coverage guarantee for arrival time inference}
\label{sec:general-inference-supplement}

Fix node $u \in \mathcal{U}_n$ and let $C^{(u)}_\epsilon \subset [n]$ be a set of possible arrival times of node $u \in \mathcal{U}_n$. We say that $C^{(u)}_{\epsilon}$ is labeling-equivariant if $C^{(u)}_\epsilon$ does not depend on the labeled representation $\mathbf{T}^*_n$ of the unlabeled shape $\text{sh}(\mathbf{T}_n)$ in the sense that 
\[
C^{(u)}_\epsilon(\mathbf{T}^*_n) = C^{(\tau(u))}_\epsilon(\tau \mathbf{T}^*_n)
\]
for any $\tau \in \text{Bi}(\mathcal{U}_n, \mathcal{U}_n)$, as the node $u$ gets relabeled $\tau(u)$ under the $\tau \mathbf{T}^*_n$ representation. With this requirement and the fact that $\text{Arr}^{(u)}_\rho(\mathbf{T}_n) = \text{Arr}^{(\tau(u))}_{\tau \circ \rho}(\mathbf{T}_n)$ for any $\tau \in \Theta(\mathcal{U}_n, \mathcal{U}_n)$,  we see that $\text{Arr}^{(u)}_\rho (\mathbf{T}_n) \in C^{(u)}_\epsilon(\mathbf{T}^*_n)$ if and only if $\text{Arr}^{(\tau(u))}_{\tau \circ \rho}(\mathbf{T}_n) \in C^{(\tau(u))}_\epsilon(\tau \mathbf{T}^*_n)$ and hence the arrival time inference problem is well-defined.

Next, we show that the credible set for the arrival time of a node has valid Frequentist coverage. Recall that for a random recursive tree $\mathbf{T}_n$, we define $\tilde{\mathbf{T}}_1, \ldots, \tilde{\mathbf{T}}_n$ as the corresponding label-randomized sequence of trees. For a given labeled tree $\tilde{\mathbf{t}}_n$ with $V(\tilde{\mathbf{t}}_n) = \mathcal{U}_n$, for a node $u \in \mathcal{U}_n$, for $\epsilon \in (0, 1)$, define $B_\epsilon^{(u)}(\tilde{\mathbf{t}}_n)$ as the smallest subset of $[n]$ such that
\[
\mathbb{P}( \Pi^{-1}(u) \in B_{\epsilon}^{(u)}(\tilde{\mathbf{t}}_n) \,|\, \tilde{\mathbf{T}}_n = \tilde{\mathbf{t}}_n) =
\sum_{t \in B_\epsilon^{(u)}(\tilde{\mathbf{t}}_n)} \mathbb{P}( u \in \tilde{\mathbf{T}}_t \text{ and } u \notin \tilde{\mathbf{T}}_{t-1} \,|\, \tilde{\mathbf{T}}_n = \tilde{\mathbf{t}}_n) \geq 1 - \epsilon,
\]
where $\Pi$ is a random bijection distributed uniformly in $\text{Bi}([n], \mathcal{U}_n)$ and independently of $\mathbf{T}_n$ such that $\Pi \mathbf{T}_n = \tilde{\mathbf{T}}_n$ and where we take $\tilde{\mathbf{T}}_0 $ as the empty set. Then, we have the following guarantee.

\begin{proposition}
Let $\mathbf{T}_n$ be a random recursive tree and let $\mathbf{T}^*_n \in \text{sh}(\mathbf{T}_n)$ be any labeled representation such that $V(\mathbf{T}^*_n) = \mathcal{U}_n$. Then, for any $\rho \in \text{Bi}([n], \mathcal{U}_n)$ such that $\rho \mathbf{T}_n = \mathbf{T}^*_n$, for any $\epsilon \in (0,1)$,
\begin{align}
\mathbb{P}( \text{Arr}^{(u)}_{\rho}(\mathbf{T}_n) \in B_\epsilon^{(u)}(\mathbf{T}^*_n) ) \geq 1 - \epsilon.
\label{eq:arrival-time-credible}
\end{align}
Moreover, $B_\epsilon^{(u)}(\cdot)$ is labeling-equivariant.
\end{proposition}

\begin{proof}
We closely follow the proof of Theorem~\ref{thm:frequentist-coverage}. 

Let $\tilde{\mathbf{t}}_n$ be a labeled tree with $V(\tilde{\mathbf{t}}_n) = \mathcal{U}_n$ and let $u \in \mathcal{U}_n$ be a node. We first show labeling-equivariance: we claim that for any $\tau \in \text{Bi}(\mathcal{U}_n, \mathcal{U}_n)$,
\[
B_\epsilon^{(u)}(\tilde{\mathbf{t}}_n) = B_\epsilon^{\tau(u)}(\tau \tilde{\mathbf{t}}_n).
\]
To see this, note that $\{ \tau \tilde{\mathbf{T}}_1, \ldots, \tau \tilde{\mathbf{T}}_n \} \stackrel{d}{=} \{\tilde{\mathbf{T}}_t, \ldots, \tilde{\mathbf{T}}_n \}$ and thus, 
\begin{align*}
\mathbb{P}( u \in \tilde{\mathbf{T}}_t \text{ and } u \notin \tilde{\mathbf{T}}_{t-1} \,|\, \tilde{\mathbf{T}}_n = \tilde{\mathbf{t}}_n) &= \mathbb{P}( u \in \tau^{-1} \tilde{\mathbf{T}}_t \text{ and } u \notin \tau^{-1} \tilde{\mathbf{T}}_{t-1} \,|\, \tau^{-1} \tilde{\mathbf{T}}_n = \tilde{\mathbf{t}}_n) \\
&= \mathbb{P}( \tau(u) \in \tilde{\mathbf{T}}_t \text{ and } u \notin \tilde{\mathbf{T}}_{t-1} \,|\, \tilde{\mathbf{T}}_n = \tau \tilde{\mathbf{t}}_n)
\end{align*}

Now let $\Pi$ be a random bijection distributed uniformly in $\text{Bi}([n], \mathcal{U}_n)$ and independently of $\mathbf{T}_n$ and let $\tilde{\mathbf{T}}_n = \Pi \mathbf{T}_n$, we have that for any $\epsilon \in (0, 1)$, 
\begin{align*}
\mathbb{P}( \text{Arr}^{(u)}_{\rho} (\mathbf{T}_n) \in B^{(u)}_{\epsilon}(\mathbf{T}^*_n) ) 
&= \mathbb{P}( \rho^{-1}(u) \in B^{(u)}_{\epsilon}(\rho \mathbf{T}_n) ) \\
&= \mathbb{P}( \Pi^{-1}(u) \in B^{(u)}_{\epsilon}(\Pi \mathbf{T}_n) \,|\, \Pi = \rho) \\
&= \mathbb{P}( \Pi^{-1}(u) \in B^{(u)}_{\epsilon}( \tilde{\mathbf{T}}_n) ) \geq 1 - \epsilon,
\end{align*}
where the last inequality follows because~\eqref{eq:arrival-time-credible} holds for every $\tilde{\mathbf{t}}_n$. 
\end{proof}

\section{Lower bound on constant $C$}
\label{sec:constant-justification}

In Section~\ref{sec:simulation}, we compare the size of our confidence sets against the bound of $C \frac{\log^2 1/\epsilon}{\epsilon^4}$ provided in \cite{bubeck2017finding} for the linear preferential attachment setting. The value of the universal constant $C$ is difficult to determine since it depends on a non-normal limiting distribution described only through its characteristics function; see \cite{janson2006limit} for more details. 

We claim however that $C \geq 0.23$. To see this, we note that when $\epsilon \leq 0.5$, any confidence set must contain at least 2 nodes since it is impossible to estimate the root with probability greater than 0.5. Therefore, with $\epsilon = 0.49$, 
\[
C \geq 2 \biggl( \frac{\log^2 (1/\epsilon)}{\epsilon^4} \biggr)^{-1} \geq 0.23.
\]

\end{document}